\documentclass[11pt, oneside]{amsart}

\usepackage{amsmath, amsfonts, amssymb}
\usepackage{srcltx, amsopn}
\usepackage{enumerate}
\usepackage{ifthen, color, hyperref}
\usepackage[cmtip,arrow]{xy}
\usepackage{pb-diagram, pb-xy}
\usepackage{overpic}

\usepackage[T1]{fontenc}

\newcommand{\showcomments}{yes}

\newsavebox{\commentbox}
\newenvironment{com}%
{\ifthenelse{\equal{\showcomments}{yes}}%
{\footnotemark
        \begin{lrbox}{\commentbox}
        \begin{minipage}[t]{1.in}\raggedright\sffamily\tiny
        \footnotemark[\arabic{footnote}]}
{\begin{lrbox}{\commentbox}}}%
{\ifthenelse{\equal{\showcomments}{yes}}%
{\end{minipage}\end{lrbox}\marginpar{\usebox{\commentbox}}}
{\end{lrbox}}}

\newcounter{ax}
\newtheorem{thm}{Theorem}[section]

\newtheorem{lem}[thm]{Lemma}

\newtheorem{cor}[thm]{Corollary}

\newtheorem{prop}[thm]{Proposition}

\newtheorem*{thmi}{Theorem}

\theoremstyle{definition}
\newtheorem{defn}[thm]{Definition}
\newtheorem{rem}[thm]{Remark}
\newtheorem*{remi}{Remark}

\newtheorem{conv}[thm]{Convention}

\newtheorem{claim*}{Claim}

 \DeclareMathOperator{\link}{Lk}

\newcommand{\field}[1]{\mathbb{#1}}

\newcommand{\naturals}{\ensuremath{\field{N}}}

\makeatletter

\newcommand{\Rmnum}[1]{\mathbf{{\expandafter\@slowromancap\romannumeral #1@}}}

 \makeatother

\let\oldmarginpar\marginpar
\renewcommand\marginpar[1]{\-\oldmarginpar[\raggedleft\footnotesize #1]%
{\raggedright\footnotesize #1}}

\newcommand\ds{\displaystyle}

\newcounter{enumitemp}

\newcommand{\core}{\mathrm{core}}
\newcommand{\blockend}{\mathrm{end}}
\def\G{{\mathcal G}}
\newcommand{\CFS}{\mathcal{CFS}}

\newcommand{\E}{\mathbb E}
\newcommand{\Pb}{\mathbb P}
\renewcommand{\Pr}{\mathbb P}
\newcommand{\Vertices}{\mathbf{Vert}}
\newcommand{\Edges}{\mathbf{Edge}}
\newcommand{\AS}{\mathcal{AS}}
\long\def\Restate#1#2#3#4{
\medskip\par\noindent
{\bf #1 \ref{#2} #3} {\it #4}\par\medskip }

\long\def\Restates#1#2#3#4#5{
\medskip\par\noindent
{\bf #1 \ref{#2} and \ref{#5} #3} {\it #4}\par\medskip }

\begin{document}

\title{Global structural properties of random graphs}

\author[J. Behrstock]{Jason Behrstock}
\address{Lehman College and The Graduate Center, CUNY, New York, New York, USA}
\curraddr{Barnard College, Columbia University, New York, New York,
USA} \email{jason@math.columbia.edu}
\thanks{\flushleft {Behrstock was supported by a Simons Fellowship and NSF grant DMS-0739392.}}

\author[V. Falgas-Ravry]{Victor Falgas-Ravry}
\address{Ume{\aa} Universitet, Ume{\aa}, Sweden}
\email{victor.falgas-ravry@umu.se}

\author[M.F. Hagen]{Mark F. Hagen}
\address{U. Cambridge, Cambridge, UK}
\email{markfhagen@gmail.com}
\thanks{\flushleft {Hagen was supported by the National Science Foundation under Grant Number NSF 1045119.}}

\author[T. Susse]{Tim Susse}
\address{U. Nebraska, Lincoln, Nebraska, USA}
\email{tsusse2@unl.edu}


\date{\today}

\maketitle

\begin{abstract}
We study two global structural properties of a graph $\Gamma$, denoted
$\AS$ and $\CFS$, which arise in
a natural way from geometric group theory.  We study these properties
in the Erd{\H o}s--R\'enyi random graph model $\G(n,p)$, proving the existence of a sharp threshold for a random graph to have the $\AS$ property asymptotically almost surely, and giving fairly tight bounds for the corresponding threshold for the $\CFS$ property.

As an application of our results, we show that for any constant $p$ and any
$\Gamma\in\G(n,p)$, the
right-angled Coxeter group $W_\Gamma$
asymptotically almost surely has quadratic divergence and thickness of order $1$,
generalizing and strengthening a result of
Behrstock--Hagen--Sisto~\cite{BehrstockHagenSisto:coxeter}.  Indeed,
we show that at a large range of densities a
random right-angled Coxeter group has quadratic divergence.
\end{abstract}

\section*{Introduction}\label{sec:results}
In this article, we consider two properties of graphs motivated by geometric group theory.  We show that these properties are typically present in random graphs.  We repay the debt to geometric group theory by applying our (purely graph-theoretic) results to the large-scale geometry of Coxeter groups.

\subsection*{Random graphs} Let $\G(n,p)$ be the random graph model on
$n$ vertices obtained by including each edge independently at random
with probability $p=p(n)$.
The parameter $p$ is often referred to as the
\emph{density} of $\G(n,p)$.  The model $\G(n,p)$ was introduced by
Gilbert~\cite{Gilbert:Random_graphs},
and the resulting random graphs are usually referred to as the
``Erd\H{o}s--R\'{e}nyi
random graphs'' in honor of Erd{\H os} and R\'enyi's seminal
contributions to the field, and we follow this convention.  We say
that a property $\mathcal{P}$ holds \emph{asymptotically almost
surely} (a.a.s.) in $\G(n,p)$ if for $\Gamma\in\G(n,p)$ we have
$\mathbb{P}(\Gamma \in
\mathcal{P})\rightarrow 1$ as $n\rightarrow \infty$.  In this paper we
will be interested in proving that certain global properties hold
a.a.s.\
in $\G(n,p)$ both for a wide range of probabilities $p=p(n)$.

A graph property is \emph{(monotone) increasing} if it is closed
under the addition of edges. A paradigm in the theory of random
graphs is that global increasing graph properties exhibit \emph{sharp thresholds} in $\G(n,p)$: for many global increasing properties $\mathcal{P}$, there is a \emph{critical density} $p_c=p_c(n)$ such that for any fixed $\epsilon>0$ if $p<(1-\epsilon)p_c$ then a.a.s.\ $\mathcal{P}$ does not hold in $\G(n,p)$, while if $p>(1+\epsilon)p_c$ then a.a.s.\ $\mathcal{P}$ holds in $\G(n,p)$. A quintessential example is the following classical result of Erd{\H o}s and R\'enyi which provides a sharp threshold for connectedness:
\begin{thmi}[Erd\H{o}s-R\'{e}nyi; \cite{ErdosRenyi3}]\label{ER:connected}
 There is a sharp threshold for connectivity of a random graph with critical density $p_c(n)=\frac{\log(n)}{n}$. 
\end{thmi}

The local
structure of the Erd{\H o}s--R\'enyi random graph is well understood,
largely due to the assumption of independence between the edges.
For example, Erd\H{o}s--R\'{e}nyi and others have obtained threshold
densities for the existence of certain subgraphs in a random graph
(see e.g.\ \cite[Theorem 1, Corollaries 1--5]{ErdosRenyi3}).  In
earlier applications of random graphs to geometric group theory, this
feature of the model was successfully exploited in order to analyze
the geometry of right-angled Artin and Coxeter groups presented by
random graphs; this is notable, for example, in the work of Charney
and Farber \cite{CharneyFarber}.  In particular, the presence of an induced square implies non-hyperbolicity of the associated
right-angled Coxeter group~\cite{CharneyFarber, ErdosRenyi3, Moussong:thesis}.

In this paper, we take a more global approach.  Earlier work
established a correspondence between some
fundamental geometric properties of right-angled Coxeter groups
and large-scale structural properties of the presentation
graph, rather than local properties such as the presence or absence
of certain specified subgraphs. The simplest of these properties is
the property of being the join of two subgraphs that are not cliques.

One large scale graph property relevant in the present context is
a property studied in~\cite{BehrstockHagenSisto:coxeter} which,
roughly, says that the graph is constructed in a particular
organized, inductive way
from joins.  In this paper, we discuss a refined version of this
property, $\CFS$, which is a slightly-modified version of a property
introduced by Dani--Thomas~\cite{DaniThomas:divcox}.  We also study
a stronger property, $\AS$, and show it is
generic in random graphs for a large range of $p(n)$, up $1-\omega(n^{-2})$.

\subsection*{$\AS$ graphs} The first class of graphs we study is the
class of \emph{augmented suspensions}, which we denote $\AS$.  A graph is an augmented suspension if it contains an induced
subgraph which is a suspension (see Section~\ref{sec:definitions} for a precise definition of this term), and any vertex which is not in that
suspension is connected by edges to at least two nonadjacent vertices
of the suspension.  

\Restates{Theorems}{thm:as_upper2}{(Sharp Threshold for $\AS$).}
{Let $\epsilon>0$ be fixed. If $p=p(n)$ satisfies $p\ge(1+\epsilon)\left(\ds\frac{\log n}{n}\right)^{\frac{1}{3}}$ and $(1-p)n^2\to\infty$,
then $\Gamma\in\G(n,p(n))$ is a.a.s.\ in $\AS$.
On the other hand, if
$p\le(1-\epsilon)\left(\ds\frac{\log{n}}{n}\right)^{\frac{1}{3}}$, then $\Gamma\in\G(n,p)$ a.a.s.\ does not lie in $\AS$.}
{thm:not_augsusp}

Intriguingly, Kahle proved  that a function similar to the critical density in
Theorem~\ref{thm:as_upper2} is the threshold for a random simplicial
complex to have vanishing second rational cohomology
\cite{Kahle:ThresholdsCohom}.

\begin{remi}[Behaviour near $p=1$]
Note that property $\AS$ is not monotone increasing, since it
requires the presence of a number of non-edges.  In particular,
complete graphs are not in $\AS$.  Thus unlike the global properties
typically studied in the theory of random graphs, $\AS$ will cease to
hold a.a.s.\ when the density $p$ is very close to $1$. In fact,
\cite[Theorem 3.9]{BehrstockHagenSisto:coxeter} shows that if
$p(n)=1-\Omega\left(\frac{1}{n^2}\right)$, then a.a.s.\ $\Gamma$ is either a clique or a clique minus a fixed number of edges whose endpoints are all disjoint. Thus, with positive probability, $\Gamma\in\AS$. However, \cite[Theorem
1]{CharneyFarber} shows that if $(1-p)n^2\to 0$ then $\Gamma$ is
asymptotically almost surely a clique, and hence not in $\AS$.

\end{remi}

%

\subsection*{$\CFS$ graphs} The second family of graphs, which we call
\emph{$\CFS$ graphs} (``Constructed From Squares''), arise naturally in geometric group theory in
the context of the large--scale geometry of right--angled Coxeter
groups, as we explain below and in
Section~\ref{sec:application_to_racg}.  A special case of these
graphs was introduced
by Dani--Thomas to study divergence in triangle-free right-angled
Coxeter groups \cite{DaniThomas:divcox}. The graphs we study are intimately related
to a property called \emph{thickness}, a feature of many key
examples in geometric group theory and low dimensional topology that is
closely related to divergence, relative hyperbolicity, and a number of
other topics.  This property is, in essence, a connectivity property
because it relies on a space being ``connected'' through sequences
of ``large'' subspaces.
Roughly speaking, a graph is $\CFS$ if it can be built inductively by chaining
(induced) squares together in such a way that each square overlaps with one of the
previous squares along a diagonal (see Section~\ref{sec:definitions} for a precise definition).
We explain in the next section how this
class of graphs generalizes $\AS$.
Our next result about genericity of $\CFS$ combines with
Proposition~\ref{prop:thick_of_order_1} below to significantly
strengthen~\cite[Theorem~VI]{BehrstockHagenSisto:coxeter}.
This result is an immediate consequence of
Theorems~\ref{thm:CFS} and~\ref{Thm: CFS lower bound}, which, in fact, establish
slightly more precise, but less concise, bounds.
\Restates{Theorems}{thm:CFS}{}{Suppose $(1-p)n^2\to\infty$ and let
    $\epsilon>0$. Then $\Gamma\in\G(n,p)$ is a.a.s.\ in $\CFS$
    whenever $p(n)>n^{-\frac{1}{2}+\epsilon}$. Conversely,
    $\Gamma\in\G(n,p)$ is a.a.s.\ not in $\CFS$
whenever $p(n)<n^{-\frac{1}{2} -\epsilon}.$}{Thm: CFS lower bound}

\noindent We actually show, in Theorem~\ref{thm:CFS}, that at densities above $5\sqrt{\frac{\log
n}{n}}$, with $(1-p)n^2\to\infty$, the random graph is a.a.s. in $\CFS$, while in Theorem~\ref{Thm: CFS lower bound} we show a random graph a.a.s. not in $\CFS$ at densities below $\frac{1}{\sqrt{n}\log{n}}$.

Theorem~\ref{thm:CFS} applies to graphs in a range strictly larger
than that in which Theorem~\ref{thm:as_upper2} holds (though our proof
of Theorem~\ref{thm:CFS} relies on Theorem~\ref{thm:as_upper2} to deal
with the large $p$ case).  Theorem~\ref{thm:CFS} combines with
Theorem~\ref{thm:not_augsusp} to show that, for densities between
$\left(\log n/n\right)^{\frac{1}{2}}$ and $\left(\log
n/n\right)^{\frac{1}{3}}$, a random graph is asymptotically almost
surely in $\CFS$ but not in $\AS$.
We also note that Babson--Hoffman--Kahle \cite{BabsonHoffmanKahle}
proved that a function of
order $n^{-\frac{1}{2}}$
appears as the threshold for simple-connectivity in the
Linial--Meshulam model for random 2--complexes \cite{LinialMeshulam}.
It would be interesting to understand whether there is a connection
between genericity of the $\CFS$ property and the topology of random
2-complexes.

Unlike our results for the $\AS$ property, we do not establish a sharp
threshold for the $\CFS$ property.  In fact, we believe that neither
the upper nor lower bounds, given in Theorem~\ref{thm:CFS} and
Theorem~\ref{Thm: CFS lower bound}, for the critical density around which $\CFS$ goes from a.a.s.\ not holding to a.a.s.\ holding are sharp.  Indeed, we believe
that there is a sharp threshold for the $\CFS$ property located at
$p_c(n)=\theta(n^{-\frac{1}{2}})$.  This conjecture is linked to the
emergence of a giant component in the ``square graph'' of $\Gamma$
(see the next section for a definition of the square graph and the
heuristic discussion after the proof of Theorem~\ref{Thm: CFS lower
bound}).

\subsection*{Applications to geometric group theory}\label{subsec:intro_coxeter}
Our interest in the structure of random graphs was sparked largely by
questions about the large-scale geometry of \emph{right-angled
Coxeter groups}. Coxeter groups were first introduced 
in~\cite{Coxeter} as a generalization of reflection groups, i.e., discrete groups generated by a specified set of reflections in Euclidean space.  A reflection group is \emph{right-angled} if the reflection loci intersect at right angles.  An abstract right-angled Coxeter group generalizes this situation: it is defined by a group presentation in which the generators are involutions and the relations are obtained by declaring some pairs of generators to commute.  Right-angled Coxeter groups (and more general Coxeter groups) play an important role in geometric group theory and are closely-related to some of that field's most fundamental objects, e.g. CAT(0) cube complexes~\cite{Davis:book,NibloReeves,HaglundWise:coxeter} and (right-angled) Bruhat-Tits buildings (see e.g.~\cite{Davis:book}).  

A right-angled Coxeter
group is determined by a
unique finite simplicial \emph{presentation graph}: the vertices correspond to the involutions generating the group, and the edges encode the pairs of generators that commute.  In fact, the presentation graph uniquely determines the right-angled Coxeter group~\cite{Muhlherr:aut_cox}.  In this paper, as an application of our results on random graphs, we continue the project of understanding large-scale geometric features of right-angled Coxeter groups in terms of the combinatorics of the presentation graph, begun in~\cite{BehrstockHagenSisto:coxeter,CharneyFarber,DaniThomas:divcox}.  Specifically, we study right-angled Coxeter groups defined by random presentation graphs,  focusing on the prevalence of two important geometric properties: \emph{relative hyperbolicity} and \emph{thickness}.

Relative hyperbolicity, in the
sense introduced by Gromov and equivalently formulated by many
others~\cite{Gromov:hyperbolic,Farb:relhyp,Bowditch:RelHyp,
Osin:RelHyp}, when it holds, is a powerful tool for studying groups. On the other hand, thickness of a finitely-generated group (more
generally, a metric space) is a property introduced by
Behrstock--Dru\c{t}u--Mosher in~\cite{BDM} as a geometric obstruction
to relative hyperbolicity and has a number of powerful geometric
applications. For example, thickness gives bounds on divergence (an important quasi-isometry invariant of
a metric space) in many different groups and spaces~\cite{BehrstockDrutu:thick2, BehrstockHagen:cubulated1,
BrockMasur:WPrelhyp, DaniThomas:divcox, Sultan:thesisarticle}.

Thickness is an inductive property: in the present context, a finitely generated group
$G$ is \emph{thick of order $0$} if and only if it decomposes as the direct
product of two infinite subgroups.  The group $G$ is
thick of order $n$ if there exists a finite collection $\mathcal H$ of
undistorted subgroups of $G$, each thick of order $n-1$, whose union
generates a finite-index subgroup of $G$ and which has the following
``chaining'' property: for each $g,g'\in G$, one can construct a
sequence $g\in g_1H_1,g_2H_2,\ldots,g_kH_k\ni g'$ of cosets, with each
$H_i\in\mathcal H$, so that consecutive cosets have infinite coarse
intersection.  Many of the best-known groups studied by geometric
group theorists are thick, and indeed thick of order $1$: one-ended
right-angled Artin groups, mapping class groups of surfaces, outer
automorphism groups of free groups, fundamental groups of
$3$--dimensional graph manifolds, etc.~\cite{BDM}.

The class of Coxeter groups contains many examples
of hyperbolic and relatively hyperbolic groups.
There is a criterion for hyperbolicity purely in terms of the
presentation graph due to
Moussong \cite{Moussong:thesis} and an algebraic criterion for relative hyperbolicity due to Caprace
\cite{Caprace:relatively_hyperbolic}.
The class of Coxeter groups
includes examples which are non-relatively hyperbolic, for instance,
those constructed by Davis--Januszkiewicz \cite{DavisJanuszkiewicz} and, also, ones
studied by Dani--Thomas \cite{DaniThomas:divcox}.
In fact, in~\cite{BehrstockHagenSisto:coxeter}, this is taken further: it is
shown that every Coxeter group is actually either thick, or hyperbolic
relative to a canonical collection of thick Coxeter subgroups. Further, there is a
simple, structural condition on the presentation
graph, checkable in polynomial time, which characterizes thickness.  This result is
needed to deduce the applications below from our graph theoretic
results.

Charney and Farber initiated the study of \emph{random
graph products} (including right-angled Artin and Coxeter groups)
using the Erd\H{o}s-R\'{e}nyi model of random
graphs~\cite{CharneyFarber}.
The structure of the group cohomology of random graph products
was obtained in \cite{DavisKahle}.
In~\cite{BehrstockHagenSisto:coxeter},
various results are proved about which random graphs have the
thickness property discussed above, leading to the conclusion that, at
certain low densities, random right-angled Coxeter groups
are relatively hyperbolic (and thus not thick),
while at higher densities, random
right-angled Coxeter groups are thick.  In this paper, we improve
significantly on one of the latter results, and also prove something
considerably more refined: we isolate not just thickness of random
right-angled Coxeter groups, but thickness of a specified order,
namely $1$:

\Restate{Corollary}{cor:high_density}{(Random Coxeter groups are thick of order 1.)}
{There exists a constant $C>0$ such that if $p\colon\naturals\rightarrow(0,1)$ satisfies
$\left(\ds\frac{C\log{n}}{n}\right)^{\frac{1}{2}}\leq p(n)\leq 1-\ds\frac{(1+\epsilon)\log{n}}{n}$ for some $\epsilon>0$, 
then the random right-angled Coxeter group $W_{G_{n,p}}$ is asymptotically almost
surely thick of order exactly $1$, and in particular has quadratic divergence.}

Corollary~\ref{cor:high_density}
significantly improves on Theorem~3.10 of~\cite{BehrstockHagenSisto:coxeter}, as
discussed in Section~\ref{sec:application_to_racg}. This theorem
follows from Theorems~\ref{thm:CFS} and \ref{thm:as_upper2}, the
latter being needed to treat the case of large $p(n)$, including the
interesting special case in which $p$ is constant.

\begin{rem}
We note that characterizations of thickness of right-angled Coxeter
groups in terms of the structure of the presentation graph appear to
generalize readily to graph products of arbitrary finite groups and,
probably, via the action on a cube complex constructed by Ruane and
Witzel in~\cite{RuaneWitzel:graph_products}, to arbitrary graph
products of finitely generated abelian groups, using appropriate modifications of the results in~\cite{BehrstockHagenSisto:coxeter}.	
\end{rem}

\subsection*{Organization of the paper} In
Section~\ref{sec:definitions}, we give the formal definitions of $\AS$
and $\CFS$ and introduce various other graph-theoretic notions we will
need.  In Section~\ref{sec:application_to_racg}, we discuss the
applications of our random graph results to geometric group theory, in
particular to right-angled Coxeter groups and more general graph
products.  In Section~\ref{sec:improved_as}, we obtain a sharp
threshold result for $\AS$ graphs.  Section~\ref{sec:improving} is
devoted to $\CFS$ graphs.  Finally Section~\ref{sec:experiments}
contains some simulations of random graphs with density near the
threshold for $\AS$ and $\CFS$.

\subsection*{Acknowledgments}
J.B.\ thanks the Barnard/Columbia Mathematics Department for their
hospitality during the writing of this paper. J.B.\ thanks Noah Zaitlen for introducing him to the joy
of cluster computing and for his generous time spent answering
remedial questions about programming. Also, thanks to Elchanan Mossel
for several interesting conversations about random graphs.  M.F.H.\ and T.S.\ thank the organizers of the 2015 Geometric Groups on the Gulf Coast conference, at which some of this work was completed.

This work benefited from several pieces of software, the
results of one of which is discussed in Section~\ref{sec:experiments}.  Some of the software
written by the authors incorporates components previously written by J.B.\
and M.F.H.\ jointly with Alessandro Sisto.  Another related useful
program was written by Robbie Lyman under the supervision of J.B.\ and
T.S.\ during an REU program supported by NSF Grant
DMS-0739392, see \cite{Lyman:honorsthesis}.
This research was supported, in part, by a grant of computer time
from the City University of New York High Performance Computing
Center under NSF Grants CNS-0855217, CNS-0958379, and ACI-1126113.

We thank the anonymous referee for their helpful comments.

\section{Definitions}\label{sec:definitions}
\begin{conv}
A \emph{graph} is a pair of finite sets $\Gamma=(V,E)$, where $V=V(\Gamma)$ is a set of \emph{vertices}, and $E=E(\Gamma)$ is a collection of pairs of distinct elements of $V$, which constitute the set of \emph{edges} of $G$. 
A \emph{subgraph} of $\Gamma$ is a graph $\Gamma'$ with $V(\Gamma')\subseteq V(\Gamma)$ and $E(\Gamma')\subseteq E(\Gamma)$; $\Gamma'$ is said to be an \emph{induced} subgraph of $\Gamma$ if $E(\Gamma')$ consists exactly of those edges from $E(\Gamma)$ whose vertices lie in $V(\Gamma')$. In this paper we focus on induced subgraphs, and we generally write ``subgraph'' to mean ``induced subgraph''. In particular we often identify a subgraph with the set of vertices inducing it, and we write $\vert \Gamma\vert$ for the \emph{order} of $\Gamma$, that is, the number of vertices it contains. A \emph{clique} of size $t$ is a complete graph
on $t\geq0$ vertices. Note this includes the degenerate case of the
empty graph on $t= 0$ vertices.
\end{conv}

\begin{defn}[Link, join]\label{defn:link_join}
Given a graph $\Gamma$, the \emph{link} of a vertex $v\in\Gamma$,
denoted $\link_\Gamma(v)$, is the subgraph spanned by the set of
vertices adjacent to $v$.  Given graphs $A,B$, the \emph{join}
$A\star B$ is the graph formed from $A\sqcup B$ by joining each
vertex of $A$ to each vertex of $B$ by an edge. A \emph{suspension} is a
join where one of the factors $A,B$ is the graph consisting of
two vertices and no edges.
\end{defn}

We now describe a family of graphs, denoted $\CFS$, which satisfy the
global structural property that they are
``constructed from squares.''

\begin{defn}[$\CFS$]\label{defn:CFS}
Given a graph $\Gamma$, let $\square(\Gamma)$ be the auxiliary graph
whose vertices are the induced $4$--cycles from $\Gamma$, with two
distinct $4$--cycles joined by an edge in $\square(\Gamma)$ if and
only if they intersect in a pair of non-adjacent vertices of $\Gamma$ 
(i.e., in a diagonal).
We refer to $\square(\Gamma)$ as the \emph{square-graph} of $\Gamma$.
A graph $\Gamma$ belongs to $\CFS$ if $\Gamma=\Gamma'\star
K$, where $K$ is a (possibly
empty) clique and $\Gamma'$ is a non-empty subgraph
such that $\square(\Gamma')$ has a connected component
$C$ such that the union of the $4$--cycles from $C$ covers all of
$V(\Gamma')$. Given a vertex $F\in\square(\Gamma)$, we refer to the
vertices in the $4$--cycle in $\Gamma$ associated to $F$ as the
\emph{support} of $F$.
\end{defn}

\begin{rem}\label{rem:cfs_dt}
    Dani--Thomas introduced \emph{component with full support}
    graphs
    in \cite{DaniThomas:divcox}, a subclass of the class of triangle-free graphs.  We note that each component with full support graph is constructed from squares, but the converse is not true.  Indeed, since we do not require our graphs to be triangle-free, our definition necessarily only counts \emph{induced} 4--cycles and allows them to intersect in more ways than in \cite{DaniThomas:divcox}. This
    distinction is relevant to the application to Coxeter groups,
    which we discuss in Section~\ref{sec:application_to_racg}.
\end{rem}

\begin{defn}[Augmented suspension]\label{defn:aj}
The graph $\Gamma$ is an \emph{augmented suspension} if it contains an
induced subgraph $B=\{w,w'\}\star \Gamma'$, where $w,w'$ are
nonadjacent and $\Gamma'$ is not a clique, satisfying the additional property
that if $v\in\Gamma-B$, then $\link_\Gamma(v)\cap \Gamma'$ is not a
clique.  Let $\AS$ denote the class of augmented suspensions.
Figure~\ref{fig:AS} shows a graph in $\AS$.
\end{defn}
\begin{rem}
Neither the $\CFS$ nor the $\AS$ properties introduced above are monotone with respect to the addition of edges. This stands in contrast to the most commonly studied global properties of random graphs.
\end{rem}

\begin{figure}[h]
\includegraphics[width=0.55\textwidth]{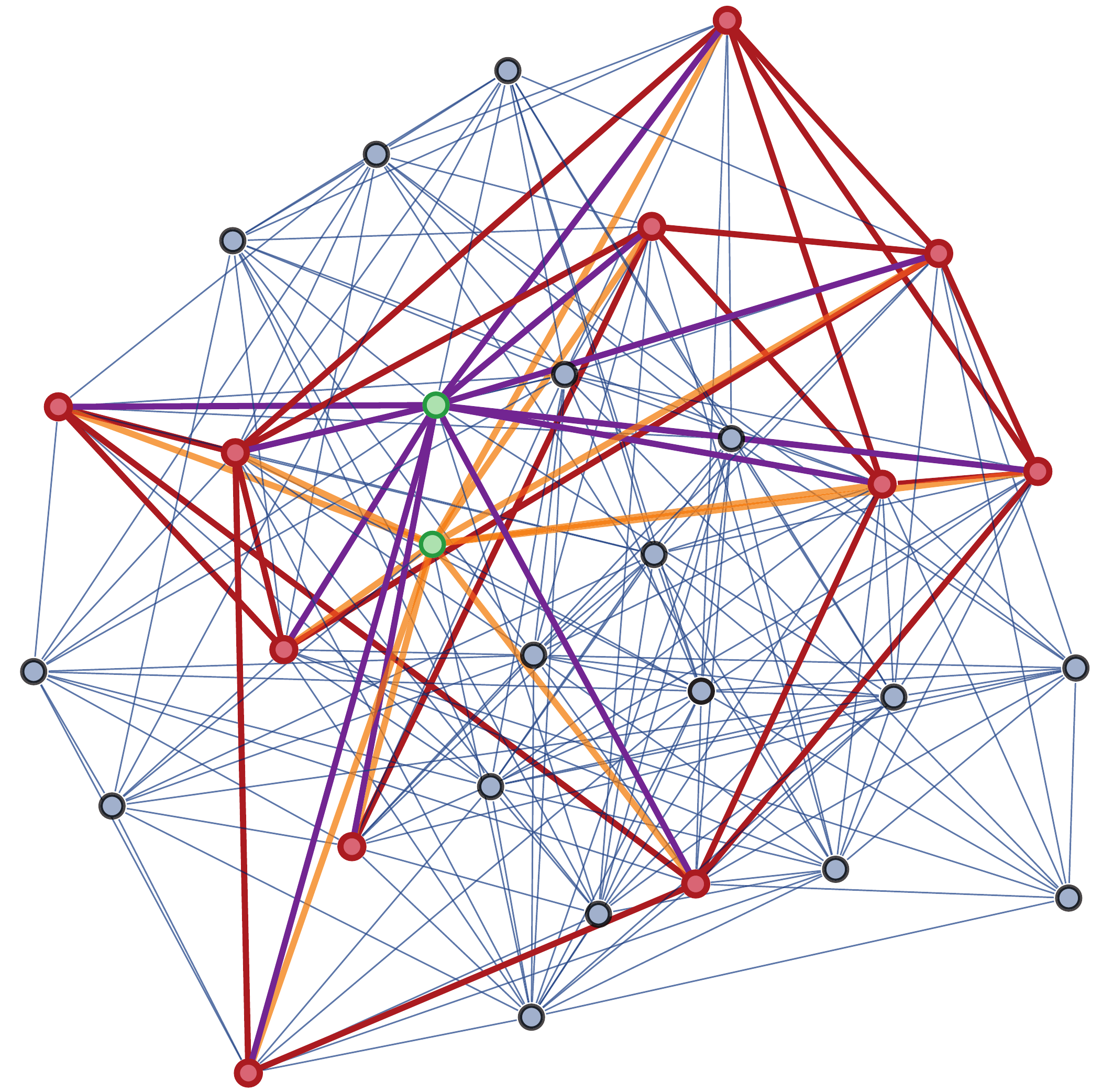}\\
\caption{A graph in $\AS$.  A block exhibiting inclusion in $\AS$ is
shown in bold; the two (left-centrally located)
\emph{ends} of the bold \emph{block} are highlighted.}\label{fig:AS}
\end{figure}

\begin{defn}[Block, core, ends]\label{defn:block_etc}
A \emph{block} in $\Gamma$  is a subgraph of the form $B(w,w')=\{w,w'\}\star
\Gamma'$ where $\{w,w'\}$ is a pair of non-adjacent vertices  and $\Gamma'\subset \Gamma$ is a subgraph of $\Gamma$ induced by a set of vertices adjacent to both $w$ and $w'$.  A block is \emph{maximal} if $V(\Gamma')=\link_{\Gamma}(w)\cap \link_{\Gamma}(w')$. Given a block
$B=B(w,w')$, we refer to the non-adjacent vertices $w,w'$ as the \emph{ends of
$B$}, denoted $\blockend(B)$, and the vertices of $\Gamma'$ as the
\emph{core of $B$}, denoted $\core(B)$.
\end{defn}

Note that $\AS\subsetneq\CFS$, indeed Theorem~\ref{thm:CFS} and
Theorem~\ref{thm:not_augsusp} show that there must exist graphs in $\CFS$ that are not in $\AS$.  Here we explain how any graph in $\AS$ is in
$\CFS$.

\begin{lem}\label{lem:AS=>CFS}
Let $\Gamma$ be a graph in $\AS$. Then $\Gamma \in \CFS$.
\end{lem}
\begin{proof}
Let $B(w,w')=\{w,w'\}\star \Gamma'$ be a maximal block in $\Gamma$ witnessing $\Gamma \in \AS$. Write $\Gamma'=A\star D$, where
$D$ is the collection of all vertices of $\Gamma'$ which are adjacent to every other vertex of $\Gamma'$. Note that $D$ induces a clique in $\Gamma$. By definition of the $\AS$ property $\Gamma'$ is not a clique, whence $A$ contains at least one pair of non-adjacent vertices. Furthermore by the definition of $D$, for every vertex $a\in A$ there exists $a' \in A$ with $\{a,a'\}$ non-adjacent. The $4$--cycles induced by $\{w,w'a,a'\}$ for non-adjacent pairs $a, a'$ from $A$ are connected in $\square(\Gamma)$. Denote the component of $\square(\Gamma)$ containing them by $C$.

Consider now a vertex $v\in \Gamma - B(w,w')$. Since $B(w,w')$ is maximal, we have that at least one of $w,w'$ is not adjacent to $v$ --- without loss of generality, let us assume that it is $w$. By the $\AS$ property, $v$ must be adjacent to a pair $a,a'$ of non-adjacent vertices from $A$. Then $\{w, a,a', v\}$ induces a $4$--cycle, which is adjacent to $\{w,w',a,a'\}\in\square(\Gamma)$ and hence lies in $C$.

Finally consider a vertex $d \in D$. If $v$ is adjacent to all vertices of $\Gamma$, then $\Gamma$ is the join of a graph with a clique containing $d$, and we can ignore $d$ with respect to establishing the $\CFS$ property. Otherwise $d$ is not adjacent to some $v\in \Gamma-B(w,w')$. By the $\AS$ property, $v$ is connected by edges to a pair of non-adjacent vertices $\{a,a'\}$ from $A$. Thus $\{d,a,a',v\}$ induces a $4$--cycle. Since (as established above) there is some $4$--cycle in $C$ containing $\{a,a',v\}$, we have that $\{d,a,a',v\} \in C$ as well. Thus $\Gamma= \Gamma'' \star K$, where $K$ is a clique and $V(\Gamma'')$ is covered by the union of the $4$--cycles in $C$, so that $\Gamma \in \CFS$ as claimed.
\end{proof}

\section{Geometry of right-angled Coxeter groups}\label{sec:application_to_racg}

If $\Gamma$ is a finite simplicial graph, the
\emph{right-angled Coxeter group} $W_\Gamma$ \emph{presented by}
$\Gamma$ is the group defined by the presentation
$$\left\langle\Vertices(\Gamma)\mid\{w^2,\,uvu^{-1}v^{-1}:u,v,w\in\Vertices(\Gamma),\{u,v\}\in\Edges(\Gamma)\right\rangle.$$
A result of M\"{u}hlherr~\cite{Muhlherr:aut_cox} shows that the correspondence $\Gamma\leftrightarrow W_\Gamma$ is bijective.  We can thus speak of ``the random right-angled Coxeter group'' --- it is the right-angled Coxeter group presented by the random graph.  (We emphasize that the above presentation provides the \emph{definition} of a right-angled Coxeter group: this definition abstracts the notion of a reflection group -- a subgroup of a linear group generated by reflections -- but infinite Coxeter groups need not admit representations as reflection groups.)

Recent papers have discussed the geometry of Coxeter groups,
especially relative hyperbolicity and closely-related quasi-isometry
invariants like divergence and thickness,
cf.~\cite{BehrstockHagenSisto:coxeter,Caprace:relatively_hyperbolic,
DaniThomas:divcox}.  In particular, Dani--Thomas introduced a
property they call \emph{having a component of full support} for
triangle-free graphs (which is exactly the triangle-free version of
$\CFS$) and they prove that under the assumption $\Gamma$ is triangle-free,
$W_\Gamma$ is thick
of order at most~$1$ if and only if it has quadratic divergence if and only if
$\Gamma$ is in
$\mathcal{CFS}$, see \cite[Theorem~1.1 and
Remark~4.8]{DaniThomas:divcox}. Since the
densities where random graphs are triangle-free are also square-free
(and
thus not $\CFS$ --- in fact, they are disconnected!), we need the following slight generalization of the
result of Dani--Thomas:

\begin{prop}\label{prop:thick_of_order_1}
Let $\Gamma$ be a finite simplicial graph.  If $\Gamma$ is in
$\mathcal{CFS}$ and $\Gamma$ does not decompose as a nontrivial
join, then $W_\Gamma$ is thick of order exactly $1$.
\end{prop}

\begin{proof}
Theorem~II of~\cite{BehrstockHagenSisto:coxeter} shows immediately
that, if $\Gamma\in\CFS$, then $W_\Gamma$ is thick, being formed by
a series of \emph{thick unions} of $4$--cycles; since each
$4$--cycle is a join, it follows that $\Gamma$ is thick of order at
most $1$.  On the other
hand,~\cite[Proposition~2.11]{BehrstockHagenSisto:coxeter} shows
that $W_\Gamma$ is thick of order at least $1$ provided $\Gamma$ is
not a join.
\end{proof}

Our results about random graphs yield:

\begin{cor}\label{cor:high_density}
There exists $k>0$ so that if $p\colon\naturals\rightarrow(0,1)$ and $\epsilon>0$ are such that $\sqrt{\frac{k\log
n}{n}}\leq p(n)\leq 1-\ds\frac{(1-\epsilon)\log{n}}{n}$ for all sufficiently large $n$, then for $\Gamma\in\G(n,p)$ the group $W_\Gamma$ is asymptotically almost
surely thick of order exactly $1$ and hence has quadratic divergence.
\end{cor}

\begin{proof}
Theorem~\ref{thm:CFS}
shows that any such $\Gamma$ is asymptotically
almost surely in $\CFS$, whence $W_\Gamma$ is thick of order at most
$1$. We emphasize that to apply this result for sufficiently
large functions $p(n)$ the proof of Theorem~\ref{thm:CFS} requires an application of Theorem~\ref{thm:as_upper2} to establish that $\Gamma$ is a.a.s.\ in $\AS$ and hence in $\CFS$ by Lemma~\ref{lem:AS=>CFS}.

By Proposition~\ref{prop:thick_of_order_1}, to show that the
order of thickness is exactly one,
it remains to rule out the possibility that $\Gamma$
decomposes as a nontrivial join.  However, this occurs if and only
if the complement graph is disconnected, which
asymptotically almost surely does not occur whenever $p(n)\le 1-\frac{(1-\epsilon)\log{n}}{n}$, by the sharp threshold for connectivity of $\G(n,1-p)$
established by Erd{\H o}s and R\'enyi in~\cite{ErdosRenyi3}.  Since this holds for $p(n)$ by assumption, we conclude that asymptotically almost surely, $W_\Gamma$ is
thick of order at least $1$.

Since $W_\Gamma$ is CAT$(0)$ and thick of order exactly $1$,
the consequence about divergence now follows from
\cite{BehrstockDrutu:thick2}.
\end{proof}

This corollary significantly generalizes Theorem~3.10
of~\cite{BehrstockHagenSisto:coxeter}, which established that, if
$\Gamma\in\G(n,\frac{1}{2})$, then $W_\Gamma$ is asymptotically
almost surely thick.  Theorem~3.10
of~\cite{BehrstockHagenSisto:coxeter} does not provide effective
bounds on the order of thickness and its proof is significantly more complicated than the proof of
Corollary~\ref{cor:high_density} given above  --- indeed, it required
several days of computation (using 2013 hardware) to establish the base case of an inductive argument.

\begin{rem}[Higher-order thickness]\label{relhyp}  A lower bound of $p(n)=n^{-\frac{5}{6}}$ for membership
in a larger class of graphs whose corresponding Coxeter groups are
thick can be found in \cite[Theorem
3.4]{BehrstockHagenSisto:coxeter}.  In fact, this argument can be
adapted to give a simple proof that a.a.s.\  thickness does not
occur at densities below $n^{-\frac{3}{4}}$.  The correct threshold for a.a.s.\  thickness is, however, unknown.
\end{rem}


\begin{rem}[Random graph products versus random presentations]
Corollary \ref{cor:high_density} and Remark \ref{relhyp} show that the
random graph model for producing random right-angled Coxeter groups
generates groups with radically different geometric properties.  This
is in direct contrast to other methods of producing random groups,
most notably Gromov's random presentation model
\cite{Gromov:asymptotic,Gromov:randomgroups} where, depending on the
density of relators, groups are either almost surely hyperbolic or
finite (with order at most $2$). This contrast speaks to the merits
of considering a random right-angled Coxeter group as a natural
place to study random groups. For instance, Calegari--Wilton
 recently showed that in the
Gromov model a random group contains many subgroups which are
isomorphic to the
fundamental group of a compact hyperbolic 3--manifold
\cite{CalegariWilton:3mfldeverywhere}; does the random
right-angled Coxeter group also contain such subgroups?

Right-angled Coxeter groups, and indeed thick ones, are closely
related to Gromov's random groups in another way.  When the parameter
for a Gromov random groups is $<\frac{1}{6}$ such a group is
word-hyperbolic \cite{Gromov:asymptotic} and acts properly and
cocompactly on a CAT(0) cube complex \cite{OW:random}.  Hence the
Gromov random group virtually embeds in a right-angled Artin
group~\cite{Agol:virtual_haken}.  Moreover, at such parameters such a random group is one-ended \cite{DGP}, whence the associated right-angled Artin group is as well. By
\cite{BehrstockCharney} this right-angled Artin group  is thick of
order 1. Since any right-angled Artin group is commensurable with
a right-angled Coxeter group \cite{DavisJanuszkiewicz}, one obtains a
thick of order $1$ right-angled Coxeter group containing the randomly presented group.
\end{rem}

\section{Genericity of $\AS$}\label{sec:improved_as}
We will use the following standard Chernoff bounds, see e.g.~\cite[Theorems~A.1.11 and A.1.13]{AlonSpencer:book}:
\begin{lem}[Chernoff bounds]\label{lem:chernoff}
Let $X_1,\ldots,X_n$ be independent identically distributed random
variables taking values in $\{0,1\}$, let $X$ be their sum, and let
$\mu=\E[X]$.  Then for any $\delta\in(0,2/3)$
$$\Pb\left(\vert X-\mu\vert \geq \delta\mu\right)\leq
2e^{-\frac{\delta^2\mu}{3}}.$$
\end{lem}
\begin{cor}\label{cor:concentration_link_sizes}
Let $\varepsilon, \delta>0$ be fixed.
\begin{enumerate}[(i)]
\item   If $p(n)\ge \left(\frac{(6+\varepsilon)\log n}{\delta^2
n}\right)^{1/2}$, then a.a.s.\ for all pairs of distinct vertices
$\{x,y\}$ in $\Gamma \in \G(n,p)$ we have  $\left\vert \vert \link_{\Gamma}(x)\cap \link_{\Gamma}(y)\vert-p^2(n-2)\right\vert < \delta p^2(n-2)$.
\item   If $p(n)\ge \left(\frac{(9+\varepsilon)\log n}{\delta^2
n}\right)^{1/3}$, then a.a.s.\ for all triples of distinct vertices
$\{x,y, z\}$ in $\Gamma \in \G(n,p)$ we have  $\left\vert \vert \link_{\Gamma}(x)\cap \link_{\Gamma}(y)\cap \link_{\Gamma}(z)\vert-p^3(n-3)\right\vert < \delta p^3(n-3)$.
\end{enumerate}
\end{cor}
\begin{proof}
For (i), let $\{x,y\}$ be any pair of distinct vertices.  For each
vertex $v \in \Gamma-\{x,y\}$, set $X_v$ to be the indicator function
of the event that $v\in \link_{\Gamma}(x)\cap \link_{\Gamma}(y)$, and
set $X=\sum_v X_v$ to be the size of $\link_{\Gamma}(x)\cap
\link_{\Gamma}(y)$.  We have $\mathbb{E}X=p^2(n-2)$ and so by the
Chernoff bounds above, $\Pr\left(\vert X-p^2(n-2)\vert\geq \delta
p^2(n-2)\right) \leq 2e^{-\frac{\delta^2p^2(n-2)}{3}}$.  Applying
Markov's inequality, the probability that there exists \emph{some} ``bad
pair'' $\{x,y\}$ in $\Gamma$ for which $\vert\link_{\Gamma}(x)\cap
\link_{\Gamma}(y)\vert$ deviates from its expected value by more than
$\delta p^2(n-2)$ is at most
\[\binom{n}{2}2e^{-\frac{\delta^2p^2(n-2)}{3}} =o(1),\]
provided $\delta^2 p^2n\ge (6+\varepsilon) \log n$ and $\varepsilon,
\delta>0$ are fixed. Thus for this range of $p=p(n)$, a.a.s.\ no such bad pair exists.  The proof of (ii) is nearly identical.
\end{proof}

\begin{lem}\label{lem:cliquesize} \begin{enumerate}[(i)]
		\item  Suppose $1-p\geq \frac{\log n}{2n}$. Then asymptotically almost surely, the order of a largest clique in  $\Gamma\in\G(n,p)$ is $o(n)$.
		\item Let $\eta$ be fixed with $0<\eta<1$. Suppose $1-p\geq \eta$. Then asymptotically almost surely, the order of a largest clique in $\Gamma \in \G(n,p)$ is $O(\log n)$.
			\end{enumerate}
\end{lem}
\begin{proof} 
For (i), set $r=\alpha n$, for some $\alpha$ bounded away from $0$. Write $H(\alpha)=\alpha\log\frac{1}{\alpha}+(1-\alpha)\log\frac{1}{1-\alpha}$. Using the standard entropy bound $\binom{n}{\alpha n}\leq e^{H(\alpha)n}$ and our assumption for $(1-p)$, we see that the expected number of $r$-cliques in $\Gamma$ is
\begin{align*}\ds\binom{n}{r}p^{\binom{r}{2}}&\le e^{H(\alpha)n}e^{\log(1-(1-p)) \left(\frac{\alpha^2 n^2}{2}+O(n)\right)}\leq \exp\left( - \frac{\alpha^2}{2} n\log n +O(n)\right)=o(1).
\end{align*}
Thus by Markov's inequality, a.a.s.\ $\Gamma$ does not contain a clique of size $r$, and the order of a largest clique in $\Gamma$ is $o(n)$.

The proof of (ii) is similar: suppose $1-p>\eta$. Then for any $r\leq n$,
\begin{align*}
\binom{n}{r}p^{\binom{r}{2}}&< n^r(1-\eta)^{r(r-1)/2}=\exp\left(r \left(\log n -\frac{r-1}{2}\log \frac{1}{1-\eta}\right)\right),
\end{align*}
which for $\eta>0$ fixed and $r-1>\frac{2}{\log (1/(1-\eta))}(1+ \log n )$ is as most $n^{-\frac{2}{\log (1/(1-\eta))}}=o(1)$. We may thus conclude as above that a.a.s.\ a largest clique in $\Gamma$ has order $O(\log n)$.
\end{proof}

\begin{thm}[Genericity of $\AS$] \label{thm:as_upper2}
    Suppose
    $p(n)\ge(1+\epsilon)\left(\ds\frac{\log{n}}{n}\right)^{\frac{1}{3}}$
    for some $\epsilon>0$ and $(1-p)n^2\to\infty$.  Then,
    a.a.s.\ $\Gamma\in\G(n,p)$
    is in $\AS$.
\end{thm}

\begin{proof}

    Let $\delta>0$ be a small constant to be specified later (the
    choice of $\delta$ will depend on $\epsilon$). By
    Corollary~\ref{cor:concentration_link_sizes} (i) for $p(n)$ in the range we are considering, a.a.s.\ all joint links have size at least $(1-\delta)p^2(n-2)$. Denote this event by $\mathcal{E}_1$. We henceforth condition on $\mathcal{E}_1$ occurring (not this only affects the values of probabilities by an additive factor of $\mathbb{P}(\mathcal{E}_1^c)=O(n^{-\varepsilon})=o(1)$). With probability $1-p^{\binom{n}{2}}=1-o(1)$, $\Gamma$ is not a clique, whence there there exist non-adjacent vertices in $\Gamma$. We henceforth assume $\Gamma\neq K_n$, and choose $v_1, v_2\in \Gamma$ which are not adjacent. Let $B$ be the maximal block associated with the pair $(v_1,v_2)$. We separate the range of $p$ into three.

\noindent\textbf{Case 1: $p$ is ``far'' from both the threshold and $1$.} Let $\alpha>0$ be fixed, and suppose $\alpha n^{-1/4}\leq p \leq  1- \frac{\log n}{2n}$. Let $\mathcal{E}_2$ be the event that
for every vertex $v\in\Gamma-B$ the set $\link_{\Gamma}(v)\cap B$
has size at least $\frac{1}{2}p^3(n-3)$. By
Corollary~\ref{cor:concentration_link_sizes}, a.a.s.\  event
$\mathcal{E}_2$ occurs, i.e., all vertices in $\Gamma-B$ have this
property.

We claim that a.a.s.\ there is no clique of order at least $\frac{1}{2}p^3(n-3)$ in $\Gamma$. Indeed, if $p<1-\eta$ for some fixed $\eta>0$, then by Lemma~\ref{lem:cliquesize} part (ii), a largest clique in $\Gamma$ has order $O(\log n)=o(p^3n)$. On the other hand, if $1-\eta <p\leq 1- \frac{\log n}{2n}$, then by Lemma~\ref{lem:cliquesize} part (i), a largest clique in $\Gamma$ has order $o(n)=o(p^3n)$. Thus in either case a.a.s.\ for \emph{every} $v\in
\Gamma-B$, $\link_{\Gamma}(v)\cap B$ is not a clique and
hence $v \in \overline{B}$, so that a.a.s.\ $\overline{B}=\Gamma$,
and $\Gamma \in \AS$ as required.

\noindent\textbf{Case 2:  $p$ is ``close'' to the threshold.} Suppose that $(1+\epsilon)\left(\ds\frac{\log{n}}{n}\right)^{\frac{1}{3}}\le
p(n)$ and $np^4\to 0$.
    Let $\vert B\vert=m+2$. By our conditioning, we have $(1-\delta)(n-2)p^2\leq m \leq (1+\delta)(n-2)p^2$. 
    The probability that a given vertex $v\in \Gamma$ is not in $\overline{B}$ is
    given by:
    \begin{equation}\label{equn:no_beam}\Pb(v\not\in\overline{B}\vert \{\vert B\vert =m\})=(1-p)^m+mp(1-p)^{m+1} +\sum_{r=2}^m \binom{m}{r}p^r(1-p)^{m-r}p^{\binom{r}{2}}.\end{equation}
    In this equation, the first two terms come from the case where $v$ is
    connected to $0$ and $1$ vertex in $B\setminus\{v_1,v_2\}$ respectively, while the third term
    comes from the case where the link of $v$ in $B\setminus\{v_1,v_2\}$ is a clique on $r\geq 2$ vertices. As we shall see, in the case $np^4\to 0$ which we are considering, the contribution from the first two terms dominates. Let us estimate their order:
    \begin{align*}(1-p)^m+mp(1-p)^{m-1} =\left(1+\frac{mp}{1-p}\right)(1-p)^m 
    &\leq \left(1+\frac{mp}{1-p}\right)e^{-mp}\\
    &\leq \left(1+\frac{(1-\delta)(1+\epsilon)^3\log{n}}{1-p}\right)n^{-(1-\delta)(1+\epsilon)^3}.\end{align*}
    Taking $\delta<1-\frac{1}{(1+\epsilon)^3}$ this
    expression is $o(n^{-1})$.

    We now treat the sum making up the remaining terms in Equation~\ref{equn:no_beam}. To do
    so, we will analyze the quotient of successive terms in the sum.
    Fixing $2\le r\le m-1$ we see:
    $$\frac{\binom{m}{r+1}p^{r+1}(1-p)^{m-r-1}p^{\binom{r+1}{2}}}{\binom{m}{r}p^r(1-p)^{m-r}p^{\binom{r}{2}}}=\frac{m-r-1}{r+1}\cdot\frac{p^{r+1}}{1-p}\le mp^{r+1}\le mp^{3}.$$
    Since $np^4\to 0$ (by assumption), this also tends to zero as $n\to\infty$.
    The quotients of successive terms in the sum thus tend to zero uniformly as $n\to \infty$, and we may bound the sum by a geometric series:
    $$\sum_{r=2}^m \binom{m}{r}p^r(1-p)^{m-r}p^{\binom{r}{2}}
    \le\binom{m}{2}p^3(1-p)^{m-2}\sum_{i=0}^{m-2}(mp^3)^i
    \le \left(\frac{1}{2}+o(1)\right)m^2p^3(1-p)^{m-2}.$$
    Now, $m^2p^3(1-p)^{m-2}=\frac{mp^2}{1-p}\cdot mp(1-p)^{m-1}$.
    The second factor in this expression was already shown to be $o(n^{-1})$, while $mp^2\le(1+\delta)np^4 \to 0$ by assumption, so the total contribution of the sum is $o(n^{-1})$.
    Thus for any value of $m$ between $(1-\delta)p^2(n-2)$ and $(1+\delta)p^2(n-2)$, the right hand side of Equation~\ref{equn:no_beam} is $o(n^{-1})$, and we conclude:
    $$\Pb(v\notin\overline{B}\vert\mathcal{E}_1)\leq o(n^{-1}).$$ Thus, by
    Markov's inequality,
    $$\Pb(\overline{B}=\Gamma)\geq \mathbb{P}(\mathcal{E}_1)\left(1-\sum_v\Pb(v\notin\overline{B}\vert\mathcal{E}_1)\right)=1-o(1),$$
    establishing that a.a.s.\ $\Gamma\in\AS$, as claimed.

\noindent\textbf{Case 3: $p$ is ``close'' to $1$.} Suppose $n^{-2}\ll (1-p)\leq \frac{\log n}{2n}$. Consider the complement of $\Gamma$, $\Gamma^c\in \G(n, 1-p)$. In the range of the parameter $\Gamma^c$ a.a.s.\ has at least two connected components that contain at least two vertices. In particular, taking complements, we see that $\Gamma$ is a.a.s.\ a join of two subgraphs, neither of which is a clique. It is a simple exercise to see that such as graph is in $\AS$, thus a.a.s.\  $\Gamma\in\AS$.
    \end{proof}

As we now show, the bound obtained in the above theorem is actually a sharp threshold.
Analogous to the classical proof of the connectivity threshold
\cite{ErdosRenyi3}, we consider vertices which are ``isolated'' from a block to prove that
graphs below the threshold strongly fail to be in $\AS$.

\begin{thm}\label{thm:not_augsusp}If
$p\le\left(1-\epsilon\right)\left(\ds\frac{\log{n}}{n}\right)^{\frac{1}{3}}$
for some $\epsilon>0$, then $\Gamma\in\G(n,p)$ is asymptotically
almost surely not in $\AS$.\end{thm}

\begin{proof} We will show that, for $p$ as hypothesized,
every block has a vertex ``isolated'' from it. Explicitly, let $\Gamma\in\G(n,p)$ and consider  $B=B_{v,w}=\link(v)\cap \link(w)\cup\{v,w\}$. Let $X(v,w)$ be
the event that every vertex of $\Gamma-B$ is connected by an edge to
some vertex of $B$. Clearly $\Gamma\in \AS$ only if the event $X(v,w)$ occurs for some pair of non-adjacent vertices $\{v,w\}$. Set $X=\bigcup_{\{v,w\}} X(v,w)$. Note that $X$ is a monotone event, closed under the addition of edges, so that the probability it occurs in $\Gamma\in \G(n,p)$ is a non-decreasing function of $p$. We now show that when $p=\left(1-\epsilon\right)\left(\log{n}/ n\right)^{\frac{1}{3}}$, a.a.s.\  $X$ does not occur, completing the proof.

Consider a pair of vertices $\{v,w\}$, and set $k=\vert B_{v,w}\vert$. Conditional on $B_{v,w}$ having this size and using the standard inequality $(1-x)\le e^{-x}$, we have that
$$\Pb(X(v,w))=(1-(1-p)^k)^{n-k}\le e^{-(n-k)(1-p)^k}.$$
Now, the value of $k$ is concentrated around its mean: by Corollary~\ref{cor:concentration_link_sizes},  for any fixed $\delta>0$ and all $\{v,w\}$, the order of $B_{v,w}$ is a.a.s.\ at most $(1+\delta)np^2$.
Conditioning on this event $\mathcal{E}$, we have that for any pair of vertices $v,w$,
$$\Pb(X(v,w)\vert \mathcal{E})\le \max_{k \leq (1+\delta)np^2}e^{-(n-k)(1-p)^k}=
e^{-(n-(1+\delta)np^2)(1-p)^{(1+\delta)np^2}}.$$
Now $(1-p)^{(1+\delta)np^2} = e^{(1+\delta)np^2\log(1-p)} $ and by
Taylor's theorem $\log(1-p)=-p+O(p^2)$, so that:
$$\Pb(X(v,w)\vert \mathcal{E})\leq e^{-n(1+O(p^2))e^{- (1+\delta)np^3(1+O(p))}} =e^{-n^{1-(1+\delta)(1-\epsilon)^3+o(1)}}.$$

Choosing $\delta<\ds\frac{1}{(1-\epsilon)^3}-1$, the expression above is
$o(n^{-2})$. Thus
\begin{align*}
\mathbb{P}(X)&\leq \mathbb{P}(\mathcal{E}^c)+\sum_{\{v,w\}}\mathbb{P}(X(v,w)\vert \mathcal{E})\\
&=o(1) +\binom{n}{2}o(n^{-2})=o(1).
\end{align*}
Thus a.a.s.\ the monotone event $X$ does not occur in $\Gamma\in\G(n,p)$ for $p=\left(1-\epsilon\right)\left(\log{n}/ n\right)^{\frac{1}{3}}$, and hence a.a.s. \ the property $\AS$ does not hold for $\Gamma\in\G(n,p)$ and $p(n)\leq \left(1-\epsilon\right)\left(\log{n}/ n\right)^{\frac{1}{3}}$.
\end{proof}

\section{Genericity of $\CFS$}\label{sec:improving}
The two main results in this section are upper and lower bounds for
inclusion in $\CFS$. These results are established in
Theorem~\ref{thm:CFS} and Theorem~\ref{Thm: CFS lower bound}.

\begin{thm}\label{thm:CFS}
If $p\colon\naturals\rightarrow(0,1)$ satisfies $(1-p)n^2\to\infty$ and $p(n)\geq 5\sqrt{\frac{\log
n}{n}}$ for all sufficiently large $n$, then a.a.s.\
 $\Gamma \in \G(n,p)$ lies in $\CFS$.\end{thm}

The proof of Theorem~\ref{thm:CFS} divides naturally into two ranges.
First of all for large $p$, namely for
$p(n)\geq 2\left(\log{n}/n\right)^{\frac{1}{3}}$, we appeal
to Theorem~\ref{thm:as_upper2} to show that a.a.s.\ a random graph
$\Gamma\in \G(n,p)$ is in $\AS$ and hence, by
Lemma~\ref{lem:AS=>CFS}, in $\CFS$. In light of our proof of
Theorem~\ref{thm:as_upper2}, we may think of this as the case when we
can ``beam up'' every vertex of the graph $\Gamma$ to a single block
$B_{x,y}$ in an appropriate way, and thus obtain a connected
component of $\square(\Gamma)$ whose support is all of $V(\Gamma)$

Secondly we have the case of ``small $p$'' where
\[5\sqrt{\frac{\log n}{n}} \leq p(n) \leq
2\left(\frac{\log{n}}{n}\right)^{\frac{1}{3}},\]
which is the focus of the remainder of the proof.  Here we construct a
path of length of order $n/\log n$ in $\square(\Gamma)$ onto which every vertex $v\in V(\Gamma)$ can be ``beamed up'' by adding
a $4$--cycle
whose support contains $v$.

This is done in the following manner: we start with an arbitrary pair
of non-adjacent vertices contained in a block $B_{0}$.
We then pick an arbitrary pair of
non-adjacent vertices in the block $B_0$ and let $B_1$ denote the
intersection of the block they define with $V(\Gamma)\setminus B_0$.
We repeat this procedure, to obtain a chain of blocks $B_0, B_1, B_2,
\ldots, B_t$, with  $t=O(n/\log n)$,  whose union contains a positive
proportion of $V(\Gamma)$, and which all belong to the same connected
component $C$ of $\square(\Gamma)$.  This common component $C$ is then
large enough that every remaining vertex of $V(\Gamma)$ can be
attached to it.  The main challenge is showing that our process of
recording which vertices are included in the support of a component
of the square graph does not die out or slow down too much, i.e., that the block
sizes $\vert B_i\vert$ remains relatively large at every stage of the
process and that none of the $B_i$ form a clique.

Having described our strategy, we now fill in the details, beginning
with the following upper bound on the probability of $\Gamma\in
\G(n,p)$ containing a copy of $K_{10}$, the complete graph on $10$
vertices.  The following lemma is a variant of \cite[Corollary
4]{ErdosRenyi3}:
\begin{lem}\label{lem:K10prob} Let $\Gamma\in\G(n,p)$. If
$p=o(n^{-\frac{1}{4}})$, then the probability that $\Gamma\in \G(n,p)$ contains a
clique with at least $10$ vertices is at most
$o(n^{-\frac{5}{4}})$.\end{lem}
\begin{proof} The expected number of copies of $K_{10}$ in $\Gamma$ is $$\ds\binom{n}{10}p^{\binom{10}{2}}\leq n^{10}p^{45}=o(n^{-5/4}).$$ The statement of the lemma then follows from Markov's inequality.
     \end{proof}

\begin{proof}[Proof of Theorem~\ref{thm:CFS}]
    As remarked above, Theorem~\ref{thm:as_upper2} proves Theorem~\ref{thm:CFS} for ``large'' $p$, so we only need to deal with the case where
    \[5\sqrt{\frac{\log n}{n}} \leq p(n) \leq 2\left(\frac{\log{n}}{n}\right)^{\frac{1}{3}}.\]
We iteratively build a chain of blocks, as follows. Let $\{x_0,y_0\}$ be a pair of non-adjacent vertices in $\Gamma$, if such a pair exists, and an arbitrary pair of vertices if not. Let $B_0$ be the block with ends $\{x_0, y_0\}$.

Now assume we have already constructed the blocks $B_0, \ldots, B_i$,
for $i\geq 0$.  Let $C_i=\bigcup_i B_i$ (for convenience we let
$C_{-1}=\emptyset$).
We will terminate the process and set $t=i$ if any of the three
following conditions
occur: $\vert \core(B_i)\vert\leq 6\log n $ or
$i\geq n/6\log n$ or $\vert V(\Gamma)\setminus C_i\vert \leq n/2$.
Otherwise, we let $\{x_{i+1},
y_{i+1}\}$ be a pair of non-adjacent vertices in $\core(B_i)$, if such
a pair exists, and an arbitrary pair of vertices from $\core(B_i)$
otherwise.  Let $B_{i+1}$ denote the intersection of the block whose
ends are $\{x_{i+1}, y_{i+1}\}$ and the set
$\left(V(\Gamma)\setminus(C_i)\right)\cup\{x_{i+1}, y_{i+1}\}$.
Repeat.

Eventually this process must terminate, resulting in a chain of blocks
$B_0, B_1, \ldots, B_t$.  We claim that a.a.s.\ both of the following
hold for every $i$ satisfying  $0 \leq i \leq t$:
\begin{enumerate}[(i)]
    \item  $\vert\core(B_i)\vert > 6\log n$; and
    \item $\{x_i,y_i\}$ is a non-edge in $\Gamma$.
\end{enumerate}
Part (i) follows from the Chernoff bound given in
Lemma~\ref{lem:chernoff}: for each $i\geq -1$ the set
$V(\Gamma)\setminus C_i$ contains at least $n/2$ vertices by
construction.  For each vertex $v\in V(\Gamma)\setminus C_i$, let
$X_v$ be the indicator function of the event that $v$ is adjacent to
both of $\{x_{i+1}, y_{i+1}\}$.  The random variables $(X_v)$ are
independent identically distributed Bernoulli random variables with
mean $p$.  Their sum $X=\sum_v X_v$ is exactly the size of the core of
$B_{i+1}$, and its expectation is at least $p^2n/2$.  Applying
Lemma~\ref{lem:chernoff}, we get that
\begin{align*}
\mathbb{P}(X<6\log n)&\leq \mathbb{P}(X\leq\frac{1}{2}\mathbb{E}X) \\
&\leq 2e^{-\left(\frac{1}{2}\right)^2 \frac{25\log n}{6n}}=2e^{-\frac{25}{24}\log n}.
\end{align*}
Thus the probability that $\vert \core(B_i)\vert <6\log n$ for some
$i$ with  $0\leq i \leq t$ is at most:
\[t 2e^{-\frac{25}{24}\log n} \leq \frac{4n}{5\log n}2e^{-\frac{25}{24}\log n}=o(1).\]
Part (ii) is a trivial consequence of part (i) and
Lemma~\ref{lem:K10prob}: a.a.s.\ $\core(B_i)$ has size at least $6\log
n$ for every $i$ with $0 \leq i \leq t$, and a.a.s.\ $\Gamma$ contains
no clique on $10 < \log n$ vertices, so that a.a.s.\ at each stage of
the process we could choose an non-adjacent pair $\{x_i,y_i\}$.

From now on we assume that both (i) and (ii) occur, and that $\Gamma$ contains no clique of size $10$. In addition, we
assume that $\vert \core(B_0)\vert < 8n^{\frac{1}{3}}(\log{n})^{\frac{2}{3}}$, which occurs a.a.s.\ by the Chernoff
bound. Since $\core(B_i)\geq 6\log n$ for every $i$,
we must have that by time $0<t\leq n/6\log n$ the process will have
terminated
with $C_t=\bigcup_{i=0}^t B_i$ supported on at least half of
the vertices of $V(\Gamma)$.
\begin{lem}\label{lem:CCFS}
    Either one of the assumptions above fails or there exists a connected component $F$ of $\square(\Gamma)$
    such that:
    \begin{enumerate}[(i)]
        \item for every $i$ with $0 \leq i \leq t$ and every pair
        of non-adjacent vertices $\{v,v'\}\in B_i$, there is a
        vertex in $F$ whose support in $\Gamma$ contains the pair
        $\{v,v'\}$; and

        \item the support in $\Gamma$ of the $4$--cycles
        corresponding to vertices of $F$ contains all of $C_t$
        with the exception of at most $9$ vertices of
        $\core(B_0)$; moreover, these exceptional vertices are
        each adjacent to all the vertices of $\core(B_0)$.

    \end{enumerate}
\end{lem}
\begin{proof}
By assumption the ends $\{x_0, y_0\}$ of $B_0$ are non-adjacent.
Thus, every pair of non-adjacent vertices $\{v,v'\}$ in $\core(B_0)$
induces a $4$--cycle in $\Gamma$ when taken together with $\{x_0,y_0\}$,
and all of these squares clearly lie in the same component $F$ of
$\square(\Gamma)$.  Repeating the argument with the non-adjacent pair
$\{x_1, y_1\}\in \core(B_0)$ and the block $B_1$, and then the
non-adjacent pair $\{x_2, y_2\}\in \core(B_1)$ and the block $B_2$,
and so on, we see that there is a connected component $F$ in
$\square(\Gamma)$ such that for every $0\leq i \leq t$, every
pair of non adjacent vertices $\{v,v'\} \in B_i$ lies in a $4$--cycle
corresponding to a vertex of $F$.  This establishes (i).

We now show that the support of $F$ contains all of $C_t$
except possibly some vertices in $B_0$.  We already established
that every pair $\{x_i,y_i\}$ is in the support of some vertex of $F$.
Suppose $v\in \core(B_i)$ for some $i>0$.  By construction,
$v$ is not adjacent to at least one of $\{x_{i-1}, y_{i-1}\}$,
say $x_{i-1}$.  Thus, $\{x_{i-1},x_i, y_i, v\}$ induces a $4$--cycle
which contains $v$ and is
associated to a vertex of $F$.  Finally, suppose $v\in \core(B_0)$.
By (i), $v$ fails to be in the support of $F$ only if $v$ is adjacent
to all other vertices of $\core(B_0)$.  Since, by
assumption, $\Gamma$ does not contain any clique of size $10$, there
are at most $9$ vertices not in the support of $F$, proving (ii).
\end{proof}

Lemma~\ref{lem:CCFS}, shows that a.a.s.\ we have a ``large'' component
$F$ in $\square(\Gamma)$ whose support contains ``many'' pairs of non-adjacent vertices.
In the last part of the proof, we use these pairs to prove that the
remaining vertices of $V(\Gamma)$ are also
supported on our connected component.

For each $i$ satisfying $0 \leq i \leq t$, consider a
a maximal collection, $M_{i}$,
of pairwise-disjoint pairs of vertices in $\core(B_i)\setminus\{x_{i+1},y_{i+1}\}$.  Set $M=\bigcup_i M_i$, and let
$M'$ be the subset of $M$ consisting of pairs, $\{v,v'\}$, for which
$v$ and $v'$ are not adjacent in $\Gamma$.  We have
\begin{align*}
\vert M \vert = \sum_{i=1}^{t} \left(\left\lfloor\frac{1}{2} \vert \core(B_i)\vert\right\rfloor -1\right)\geq \frac{\vert C_t \vert}{2} -2t\geq \frac{n}{4}(1-o(1)).
\end{align*}
The expected size of $M'$ is thus $(1-p)n(\frac{1}{4}-o(1))=\frac{n}{4}(1-o(1))$, and by the Chernoff bound from Lemma~\ref{lem:chernoff} we have
\begin{align*}
\mathbb{P}(\vert M'\vert \leq \frac{n}{5})&\leq 2e^{-\left(\frac{1}{5}+o(1)\right)^2 \frac{(1-p)n}{12}}\leq e^{-\left(\frac{1}{300}+o(1)\right)n},
\end{align*}
which is $o(1)$. Thus a.a.s.\ $M'$ contains at least $n/5$ pairs, and
by Lemma~\ref{lem:CCFS} each of these lies in some $4$--cycle of $F$.
We now show that we can ``beam up'' every vertex not yet supported on
$F$ by a $4$--cycle using a pair from $M'$. By construction we have
at most $n/2$ unsupported vertices from $V(\Gamma)\setminus C_t$ and
at most $9$ unsupported vertices from $\core(B_0)$.

Assume that $\vert M' \vert \geq n/5$. Fix a vertex $w\in
V(\Gamma)\setminus C_t$. For each pair $\{v,v'\}\in M'$, let
$X_{v,v'}$ be the event that $w$ is adjacent to both $v$ and $v'$.
We now observe that if $X_{v,v'}$ occurs for some pair $\{v,v'\}\in M'\cap
\core(B_i)$, then $w$ is supported on $F$.
By construction, $w$ is not
adjacent to at least one of $\{x_i, y_i\}$, let us
say without loss of generality $x_i$. Hence, $\{x_i,v,v',w\}$ is an
induced
$4$--cycle in $\Gamma$ which contains $w$ and which
corresponds to a vertex of $F$.

The probability that $X_{v,v'}$ fails to happen for every pair $\{v,v'\}\in M'$ is exactly
\begin{align*}
(1-p^2)^{\vert M' \vert}\leq (1-p^2)^{n/5}\leq e^{-\frac{p^2n}{5}}=e^{-5 \log n}.
\end{align*}
Thus the expected number of vertices $w \in V(\Gamma)\setminus C_t$
which fail to be in the support of  $F$ is at most $\frac{n}{2}e^{-5\log n}=o(1)$, whence by Markov's inequality a.a.s.\ no such bad vertex $w$ exists.

Finally, we deal with the possible $9$ left-over vertices $b_1, b_2,
\ldots b_9$ from $\core(B_0)$ we have not yet supported.  We observe
that since $\core(B_0)$ contains at most $8n^{\frac{1}{3}}(\log{n})^{\frac{2}{3}}$ vertices
(as we are assuming and as occurs a.a.s.\ , see the discussion before Lemma~\ref{lem:CCFS} ), we do not stop the process with
$B_0$, $\core(B_1)$ is non-empty and contains at least $6\log n$
vertices. As stated in Lemma~\ref{lem:CCFS}, each unsupported vertex
$b_i$ is adjacent to all other vertices in $\core(B_0)$, and in
particular to both of $\{x_1, y_1\}$. If $b_i$ fails to be adjacent
to some vertex $v\in \core(B_1)$, then the set $\{b_i, x_1, y_1, v\}$
induces a $4$--cycle corresponding to a vertex of $F$ and whose
support contains $b_i$. The probability that there is some $b_i$ not
supported in this way is at most
\begin{align*}
9 \mathbb{P}(b_i \textrm{ adjacent to all of } \core(B_1))&=9 p^{6\log n}=o(1).
\end{align*}
Thus a.a.s.\ we can ``beam up'' each of the vertices $b_1, \ldots
b_9$ to $F$ using a vertex $v\in \core(B_1)$, and the support of the
component $F$ in the square graph $\square(\Gamma)$ contains all vertices of $\Gamma$. This shows that a.a.s.\ $\Gamma \in \CFS$, and concludes the proof of the theorem.
\end{proof}

\begin{rem} The constant $5$ in Theorem~\ref{thm:CFS} is not optimal,
and indeed it is not hard to improve on it slightly, albeit at the
expense of some tedious calculations.  We do not try to obtain a
better constant, as we believe that the order of the upper bound we
have obtained is not sharp.  We conjecture that the actual threshold
for $\CFS$ occurs when $p(n)$ is of order $n^{-1/2}$ (see the
discussion below Theorem~\ref{Thm: CFS lower bound}), but a proof of
this is likely to require significantly more involved and
sophisticated arguments than the present paper.
\end{rem}

A simple lower bound for the emergence of the $\CFS$ property can be
obtained from the fact that if $\Gamma\in\CFS$, then $\Gamma$ must
contain at least $n-3$ squares; if $p(n)\ll n^{-\frac{3}{4}}$, then by
Markov's inequality a.a.s.\ a graph in $\G(n,p)$ contains fewer than
$o(n)$ squares and thus cannot be in $\CFS$.  Below, in
Theorem~\ref{Thm: CFS lower bound}, we prove a
better lower bound, showing that the order of the upper bound
we proved in Theorem~\ref{thm:CFS} is not off by a factor of more than
$(\log n)^{3/2}$.

\begin{lem}\label{lem:order} Let $\Gamma$ be a graph and let
$C$ be the subgraph of $\Gamma$ supported on a given connected component of $\square(\Gamma)$.
Then there exists an ordering $v_1<v_2< \cdots <v_{\vert
C\vert}$ of the vertices of $C$ such that for all $i\ge3$,
$v_i$ is adjacent in $\Gamma$ to at least two vertices
preceding it in the order.
\end{lem}

\begin{proof}
As $C$ is a component of $\square(\Gamma)$, it contains at
least one induced $4$--cycle.  Let $v_1, v_2$ be a pair of
non-adjacent vertices from such an induced $4$--cycle.  Then
the two other vertices $\{v_3,v_4\}$ of the $4$--cycle are
both adjacent in $\Gamma$ to both of $v_{1}$ and $v_{2}$. If this
is all of $C$, then we are done. Otherwise, we know that each
$4$--cycle in $C$ is ``connected'' to the cycle $F=
\{v_{1},v_{2},v_{3},v_{4}\}$ via a sequence of induced
$4$--cycles pairwise intersecting in pairwise non-adjacent
vertices. In particular, there is some such $4$--cycle whose
intersection with $F$ is either a pair of non-adjacent vertices
in $F$ or three vertices of $F$;
either way, we may add the new vertex next in the
order.

Continuing in this way and using the fact that the number of
vertices not yet reached is a monotonically decreasing set of
positive integers,  the lemma follows.
\end{proof}

\begin{prop}\label{prop: no giant in square(Gamma)}
Let $\delta>0$. Suppose $p\leq\frac{1}{\sqrt{n}\log n}$.  Then a.a.s.\ for
$\Gamma\in\G(n,p)$, no component of $\square(\Gamma)$ has support containing
more than $4\log n$ vertices of $\Gamma$.
\end{prop}

\begin{proof} Let $\delta>0$.
Let $m=\left\lceil \min\left(4\log n, 4\log
\left(\frac{1}{p}\right) \right)\right\rceil$, with $p\leq 1/\left(\sqrt{n}\log n\right)$.
We shall show that a.a.s.\ there is
no ordered $m$--tuple of vertices $v_1<v_2<\cdots <v_m$ from
$\Gamma$ such that for every $i\ge 2$ each vertex $v_i$ is
adjacent to at
least two vertices from $\{v_j: \ 1\leq j <i\}$.
By Lemma~\ref{lem:order}, this is enough to establish our claim.

Let $v_1<v_2<\cdots < v_m$ be an arbitrary ordered $m$--tuple
of vertices from $V(\Gamma)$.  For $i\geq 2$, let $A_i$ be the
event that $v_i$ is adjacent to at least two vertices in the
set
$\{v_j: \ 1\leq j <i\}$.  We have:
\begin{equation}\label{equn:prob(A_i)}
\Pr(A_i)= \sum_{j=2}^{i-1}\binom{i-1}{j}p^j(1-p)^{i-j-1}.
\end{equation}
As in the proof of Theorem~\ref{thm:not_augsusp} we consider the
quotients of successive terms in the sum to show that its
order is given by the term $j=2$. To see this, observe:
$$\frac{\binom{i-1}{j+1}p^{j+1}(1-p)^{i-j-2}}{\binom{i-1}{j}p^j(1-p)^{i-j-1}}
\le\frac{i-j-1}{j+1}\cdot \frac{p}{1-p}< mp$$
where the final inequality holds for $n$ sufficiently large and $p=p(n)$ satisfying our assumption.
Since $m=O\left(\log n\right)$ and $p=o(n^{-1/2})$ we have, again for $n$ large enough, that $mp=o(1)$,
and we may bound the sum in equation (\ref{equn:prob(A_i)}) by a geometric series to obtain the bound:
\begin{align*}
\Pr(A_i)&= \binom{i-1}{2}p^2(1-p)^{i-3}(1+O(mp))\\
&\leq \frac{(i-1)^2}{2} p^2(1+O(mp)).
\end{align*}
Now let $A=\bigcap_{i=1}^m A_i$. Note that the events
$A_i$ are mutually independent, since they are determined by disjoint edge-sets.
Thus we have:
\begin{align*}
\Pr(A)=\prod_{i=3}^{m} \Pr(A_i)&\leq \prod_{i=3}^m \left(\frac{(i-1)^2}{2}p^2(1+O(mp))\right)\\
&= \frac{((m-1)!)^2p^{2m-4} }{2^{m-2}}(1+O(m^2p)),
\end{align*}
where in the last line we used the equality
$(1+O(mp))^{m-2}
=1+O(m^2p)$ to bound the
error term.  Thus we have that the expected number $X$
of ordered $m$--tuples of vertices from $\Gamma$ for
which $A$ holds is at most:
\begin{align*}
\mathbb{E}(X)=\frac{n!}{(n-m)!}\Pb(A)&\leq n^{m} 4e^2\left(\frac{m^2p^{2-\frac{4}{m}}}{e^22}\right)^m (1+O(m^2p))\\
&=4e^2\left(\frac{nm^2p^{2-4/m}}{2e^2}\right)^m (1+O(m^2p)),
\end{align*}
where in the first line we used the inequality $(m-1)!\leq e (m/e)^{m}$.
We now consider the quantity
\[f(n,m,p)= \frac{nm^2 p^{2-4/m}}{2e^2}\]
which is raised to the $m^{\textrm{th}}$ power in the inequality above. We claim that $f(n,m,p)\leq e^{-1+\log 2+o(1)}$. We have two cases to consider:\\

\noindent\textbf{Case 1: $m=\lceil 4\log n\rceil $.} Since $4\log n \leq 4\log(1/p)$, we deduce that $p\leq n^{-1}$. Then
\begin{align*}
f(n,m,p)&= \frac{n (4\log n)^2 p^{2-o(1)}}{2e^2}\leq n^{-1+o(1)}\leq e^{-1+\log 2 +o(1)}.
\end{align*}
\\
\noindent\textbf{Case 2: $m=\lceil 4\log(1/p)\rceil $.}
First, note that
$p^{-4/m}=\exp\left(\frac{4\log(1/p)}{\lceil 4 \log 1/p\rceil}\right)\leq e$.
Also, for $p$ in the range $[0, n^{-1/2}(\log
n)^{-1}]$ and $n$ large enough, $p^2(\log(1/p))^2$ is
an increasing function of $p$ and is thus at most:
$$\frac{1}{n(\log n)^2} \left(\frac{1}{2}\log
n\right)^2\left(1+O\left(\frac{\log \log n}{\log
n}\right)\right)= \frac{1}{4} n^{-1} (1+o(1)).$$
Plugging this into the expression for $f(n,m,p)$, we
obtain:
\begin{align*}
f(n,m,p)&= (1+o(1))\frac{16n (\log (1/p))^2p^{2-4/m}}{2e^2}\\
&\leq (1+o(1))\frac{2}{e}= e^{-1+\log 2+o(1)} .
\end{align*}

Thus, in both cases (1) and (2) we have $f(n,m,p)\leq
e^{-1 +\log 2+o(1)}$, as claimed, whence
\begin{align*}
\mathbb{E}(X)&\leq 4e^2(f(n,m,p))^m(1+O(m^2p)) \leq  4e^2 e^{-(1-\log 2)m +o(m)}(1+O(m^2p))=o(1).
\end{align*}
It follows from Markov's inequality that the
non-negative, integer-valued random variable $X$ is
a.a.s.\ equal to $0$.  In other words, a.a.s.\ there
is no ordered $m$--tuple of vertices in $\Gamma$ for
which the event $A$ holds and, hence by
Lemma~\ref{lem:order}, no component in
$\square(\Gamma)$ covering more than $m\leq 4\log n$
vertices of $\Gamma$.
\end{proof}

\begin{thm}\label{Thm: CFS lower bound} Suppose
$p\leq\frac{1}{\sqrt{n} \log n}$.  Then a.a.s.\
$\Gamma\in\G(n,p)$ is not in $\CFS$.\end{thm}

\begin{proof}
To show that $\Gamma\not\in\CFS$, we first show that, for
$p\leq \frac{1}{\sqrt{n} \log n}$, a.a.s.\ there is no
non-empty clique $K$ such that $\Gamma=\Gamma' \star
K$.  Indeed the standard Chernoff bound guarantees that
we have a.a.s.\ no vertex in $\Gamma$ with degree
greater than $\sqrt{n}$.  Thus to prove the theorem,
it is enough to show that a.a.s.\ there is no
connected component $C$ in $\square(\Gamma)$
containing
all the vertices in $\Gamma$.  Theorem~\ref{prop: no
giant in square(Gamma)} does this by establishing the
stronger bound that a.a.s.\ there is no connected component $C$
covering more than $4\log n$ vertices.
\end{proof}
While Theorem~\ref{Thm: CFS lower bound} improves on the trivial lower bound of $n^{-3/4}$, it is still off from the upper bound for the emergence of the $\CFS$ property established in Theorem~\ref{thm:CFS}. It is a natural question to ask where the correct threshold is located.

\begin{rem}
    We strongly believe that there is a sharp threshold for the $\CFS$
    property analogous to the one we established for the $\AS$ property.
    What is more, we believe this threshold should essentially coincide
    with the threshold for the emergence of a giant component in the
    auxiliary square graph $\square(\Gamma)$.  Indeed, our arguments in
    Proposition~\ref{prop: no giant in square(Gamma)} and
    Theorems~\ref{thm:CFS} both focus on bounding the growth of a
    component in $\square(\Gamma)$.  Heuristically, we would expect a
    giant component to emerge in $\square(\Gamma)$ to emerge for
    $p(n)=cn^{-1/2}$, for some constant $c>0$, when the expected number of
    common neighbors of a pair of non-adjacent vertices in $\Gamma$ is
    $c^2$, and thus the expected number of distinct vertices in
    $4$--cycles which
    meet a fixed $4$--cycle in a non-edge is $2c^2$.  What the
    precise value of $c$ should be
    is not entirely clear (a branching process heuristics suggests $\sqrt{\ds\frac{\sqrt{17}-3}{2}}$ as a possible value, see Remark~\ref{rem: CFS speculation}), however, and the dependencies in the
    square graph make its determination a delicate matter.
\end{rem}

\section{Experiments}\label{sec:experiments}
Theorem~\ref{thm:as_upper2} and Theorem~\ref{thm:not_augsusp} show
that
$\left(\log n/n\right)^{\frac{1}{3}}$ is a sharp
threshold for the family $\AS$ and Theorem~\ref{thm:CFS} shows
that $n^{-\frac{1}{2}}$ is the
right order of magnitude of the threshold for $\CFS$.  Below we provide some empirical
results on the behaviour of random
graphs near the threshold for $\AS$ and the conjectured threshold for
$\CFS$. We also compare our experimental data with
analogous data at the connectivity threshold.  Our experiments are based on various algorithms that we implemented in $\texttt{C++}$; the source
code is available from the authors\footnote{All source code and data at
\url{www.wescac.net/research.html} or \url{math.columbia.edu/~jason} .}.

We begin with the observation that computer simulations
of $\AS$ and $\CFS$ are tractable.
%
Indeed, it is easily seen that there are polynomial-time algorithms
for deciding whether a given graph is in $\AS$ and/or $\CFS$.  Testing
for $\AS$ by examining each block and determining whether it witnesses
$\AS$ takes $O(n^5)$ steps, where $n$ is the number of vertices.  The
$\CFS$ property is harder to detect, but essentially reduces to
determining the component structure of the square graph.  The square
graph can be produced in polynomial time and, in polynomial time,
one can find its connected components and check the support of these
components in the original graph.

Using our software, we tested
random graphs in $\G(n,p)$ for membership in $\AS$ for
$n\in\{300+100k\mid0\leq k\leq12\}$ and
$\{p(n)=\alpha\left(\log
n/n\right)^{\frac{1}{3}}\mid \alpha=0.80+0.1k,\,0\leq k\leq 9\}$.
For each pair $(n,p)$ of this type, we generated $400$ random graphs
and tested each for membership in $\AS$.  (This number of tests
ensures that, with probability approximately $95\%$, the measured
proportion of $\AS$ graphs is within at most
$0.05$ of the actual proportion.)
The results are summarized
in Figure~\ref{fig:as_picture}.  The data suggests that, fixing $n$,
the probability that a random graph is in $\AS$ increases
monotonically in the range of $p$ we are considering, rising sharply from almost zero to almost one.
\begin{figure}[h]
\begin{overpic}[width=0.6\textwidth]{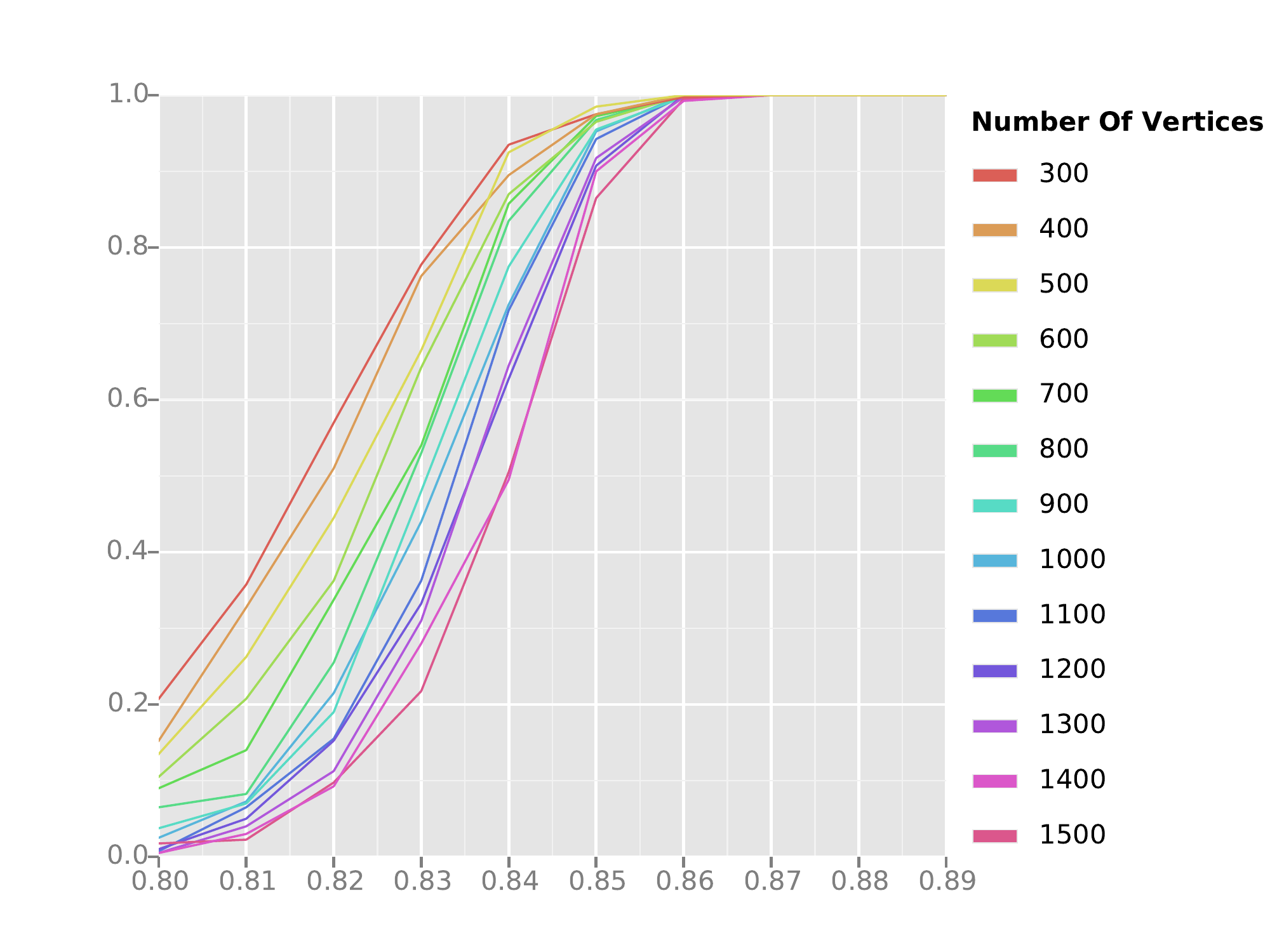}
\put(44,1){$\alpha$}
\put(-7,36){$\Pb(\AS)$}
\end{overpic}
\caption{Experimental prevalence of $\AS$ at density $\alpha\left(\frac{\log n}{n}\right)^{\frac{1}{3}}$.}\label{fig:as_picture}
\end{figure}

In Figure~\ref{fig:cfs_prevalence_PT}, we display the results of testing
random graphs for membership in $\CFS$ for
$n\in\{100+100k\mid 0\leq k\leq 15\}$ and
$\{p(n)=\alpha n^{-\frac{1}{2}} \mid \alpha=0.700+0.025k,\,0\leq
k\leq 8\}$.
For each pair $(n,p)$ of this type, we generated $400$ random graphs
and tested each for membership in $\CFS$.  (This number of tests
ensures that, with probability approximately $95\%$, the measured
proportion of $\AS$ graphs is within at most
$0.05$ of the actual proportion.) The data suggests that, fixing $n$,
the probability that a random graph is in $\CFS$ increases
monotonically in considered range of $p$: rising sharply from almost
zero to almost one inside a narrow window.

\begin{figure}[h]
\begin{overpic}[width=0.6\textwidth]{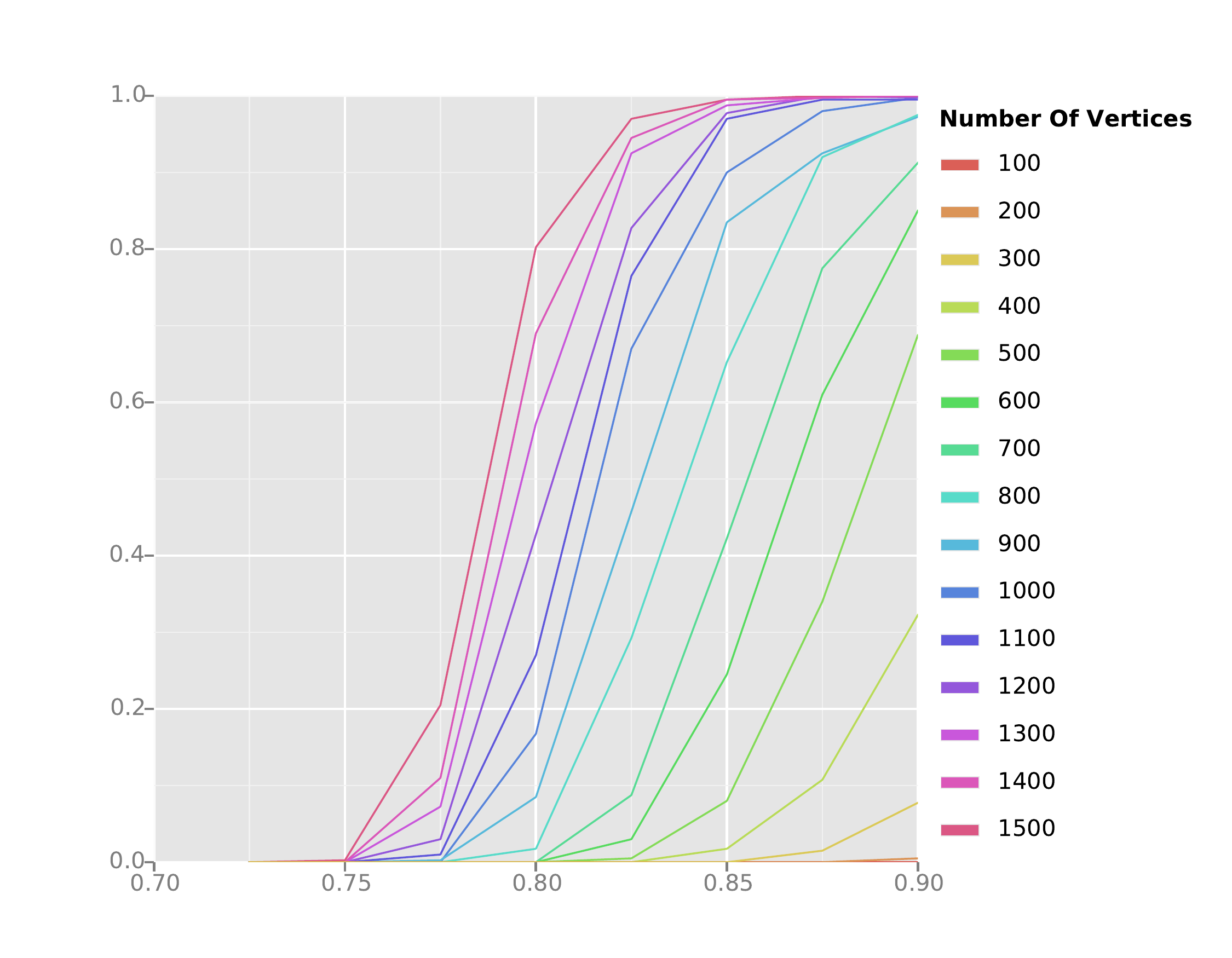}
\put(42,2){$\alpha$}
\put(-9,38){$\Pb(\CFS)$}
\end{overpic}
\caption{Experimental prevalence of $\CFS$
membership at density
$\alpha n^{-\frac{1}{2}}$.}\label{fig:cfs_prevalence_PT}
\end{figure}

\begin{rem}[Block and core sizes]\label{rem:block_core_size}
For each graph $\Gamma$ tested, the $\AS$ software also keeps track of how
many nonadjacent pairs $\{x,y\}$ --- i.e., how many blocks --- were
tested before finding one sufficient to verify membership in $\AS$; if
no such block is found, then all non-adjacent pairs have been tested
and the graph is not in $\AS$.\footnote{A set of such data comes with the source code, and
more is available upon request.}  At densities near the
threshold, this number of blocks is generally very large compared to
the number of blocks tested at densities above the
threshold.  For example, in one instance with
$(n,\alpha)=(1000,0.89)$, verifying that the graph was in $\AS$ was
accomplished after testing just $86$ blocks, while at $(1000,0.80)$, a
typical test checked all $422961$ blocks (expected number: $423397$)
before concluding that the graph is not $\AS$.  At the same
$(n,\alpha)$, another test found that the graph was in $\AS$, but only
after $281332$ tests.  This data is consonant with the spirit of our
proofs of Theorem~\ref{thm:not_augsusp} and
Theorem~\ref{thm:as_upper2}: in the former case, we showed that no
``good'' block exists, while in the latter we show that every block is
good.  We believe that right at the threshold we should have some
intermediate behavior, with the expected number of ``good'' blocks
increasing continuously from $0$ to $(1-p)\binom{n}{2}(1-o(1))$.

What is more, we expect that the more precise threshold for the $\AS$
property, coinciding with the appearance of a single ``good'' block,
should be located ``closer'' to our lower bound than to our upper one,
i.e., at $p(n)= \left((1-\epsilon)\log n/n\right)^{1/3}$, where $\epsilon(n)$
is a sequence of strictly positive real numbers tending to $0$ as
$n\rightarrow \infty$ (most likely decaying at a rate just faster than $(\log n)^{-2}$, see below).  Our experimental data, which exhibit a steep rise
in $\Pr(\AS)$ strictly before the value $\alpha$ hits one, gives some support to this guess.

Finally, our observations on the number of blocks suggests a natural
way to understand the influence of higher-order terms on the emergence of the $\AS$ property: at exactly the threshold for $\AS$,
the event $E(v,w)$ that a pair of non-adjacent vertices $\{v,w\}$
gives rise to a ``good'' block is rare and, despite some mild
dependencies, the number $N$ of pairs $\{v,w\}$ for which $E(v,w)$
occurs is very likely to be distributed approximatively like a Poisson
random variable. The probability $\Pr(N\geq 1)$ would then be a very
good approximation for $\Pr(\AS)$. ``Good'' blocks would thus play a role for the emergence of the $\AS$ property in random graphs analogous to that of isolated vertices for connectivity in random graphs.

When $p=\left((1-\epsilon) \log n/n\right)^{\frac{1}{3}}$, the
expectation of $N$ is roughly $ne^{-n^{\epsilon}(1-\epsilon)\log n}$.
This expectation is $o(1)$ when $0<\epsilon(n)=\Omega(1/n)$ and is
$1/2$ when $\epsilon(n)=(1+o(1)) \log 2/(\log n)^2$.  This suggest
that the emergence of $\AS$ should occur when $\epsilon(n)$ decays
just a little faster than $(\log n)^2$.


\end{rem}

\begin{rem}\label{rem: CFS speculation}
    Our data suggests that the prevalence of $\CFS$ is closely
    related to the emergence of a giant component in the square
    graph.  Indeed, below the experimentally observed threshold
    for
    $\CFS$, not only is the support of
    the largest component in the square graph not
    all of $\Gamma \in \G(n,p)$, but
    in fact the size of the support of the largest component is an
    extemely small proportion of
    the vertices (see Figure~\ref{fig:cfs_largecomponent_PT}).

    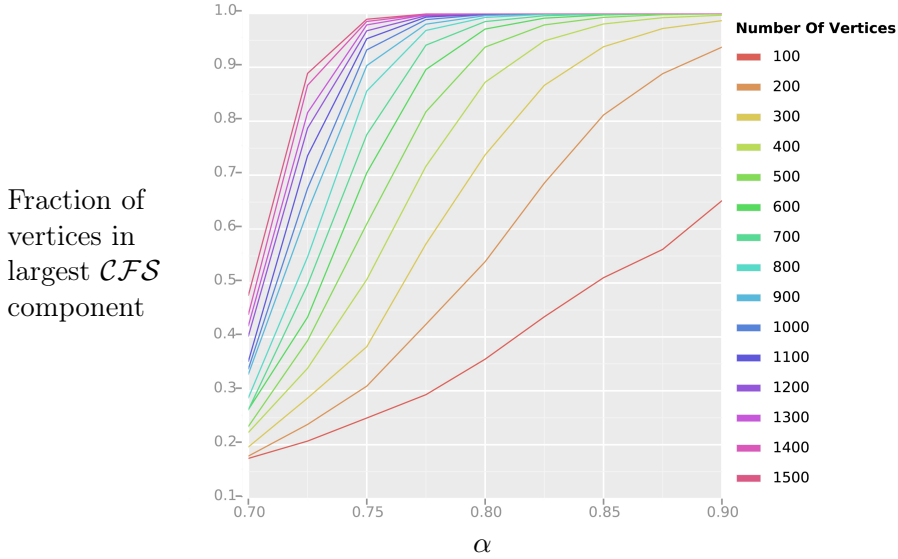
\begin{figure}[h]
    \begin{overpic}[width=0.6\textwidth]{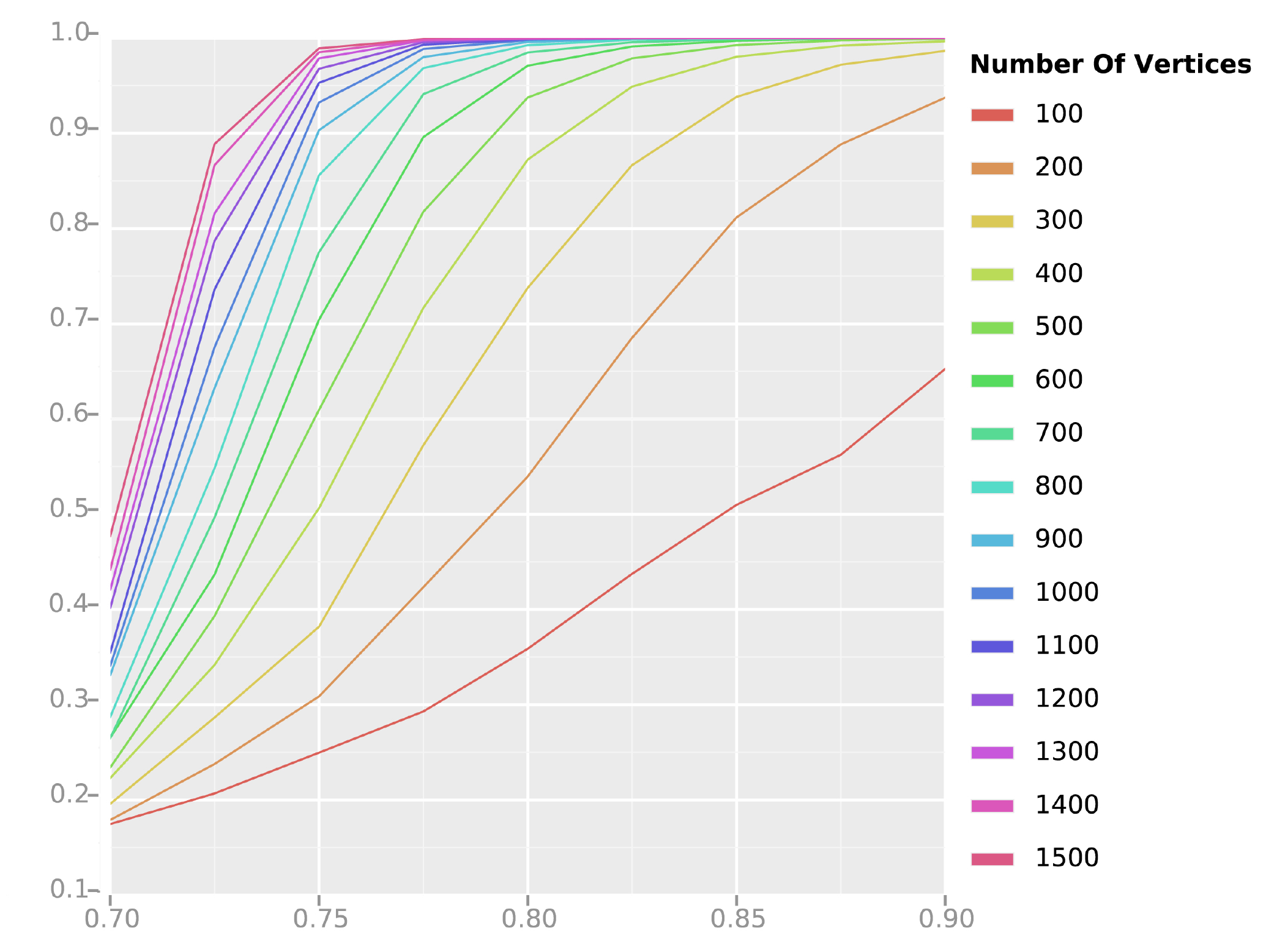}
    \put(40,-3){$\alpha$}
    \put(-25,45){Fraction of}
    \put(-25,40){vertices in}
    \put(-25,35){largest $\CFS$}
    \put(-25,30){component}
    \end{overpic}
    \caption{Fraction of vertices supporting the largest
    $\CFS$--subgraph
    at density
    $\alpha n^{-\frac{1}{2}}$.}\label{fig:cfs_largecomponent_PT}
    \end{figure}

    In the Erd{\H o}s--R\'enyi random graph, a giant component
    emerges when $p$ is around $1/n$, i.e., when vertices begin to
    expect at least one neighbor; this corresponds to a paradigmatic
    condition of expecting at least one child for survival of a
    Galton--Watson process (see~\cite{BollobasRiordan2012} for a
    modern treatment of the topic).  The heuristic observation
    that when $p=cn^{-1/2}$ the vertices of a diagonal $e$ in a fixed $4$-cycle $F$ expect to be adjacent to $cn^2$ vertices outside the $4$-cycle, giving rise to an expected $\binom{cn^2+2}{2}-1$ new $4$-cycles connected to $F$ through $e$ in $\square(\Gamma)$ suggests that $c=\sqrt{\ds\frac{\sqrt{17}-3}{2}}\approx 0.7494$ could be a reasonable guess for the location of the threshold for the $\CFS$ property.
    Our data,
    although not definitive,
    appears somewhat supportive of this guess: see
    Figure~\ref{fig:cfs_largecomponent_PT} which is based on the
    same underlying data set as Figure~\ref{fig:cfs_prevalence_PT}.

    We note that unlike an Erd{\H o}s--R\'enyi random graph, the
    square graph $\square(\Gamma)$ exhibits some strong local
    dependencies, which may make the determination of the exact
    location of its phase transition a much more delicate affair.
    \end{rem}

\begin{rem}
    For comparison with the threshold data for $\AS$ and $\CFS$, we include below a similar figure of experimental
    data for connectivity for $\alpha$ from $0.8$ to $1.4$, where $p=\frac{\alpha\log n}{n}$. Given what we know about the thresholds for connectivity and the $\AS$ property, this last set of data together with Figure~\ref{fig:as_picture} should serve as a warning not to draw overly strong conclusions: the graphs we tested are sufficiently large for the broader picture to emerge, but probably not large enough to allow us to pinpoint the exact location of the threshold for $\CFS$.
\end{rem}

\begin{figure}[h]
\begin{overpic}[width=0.6\textwidth]{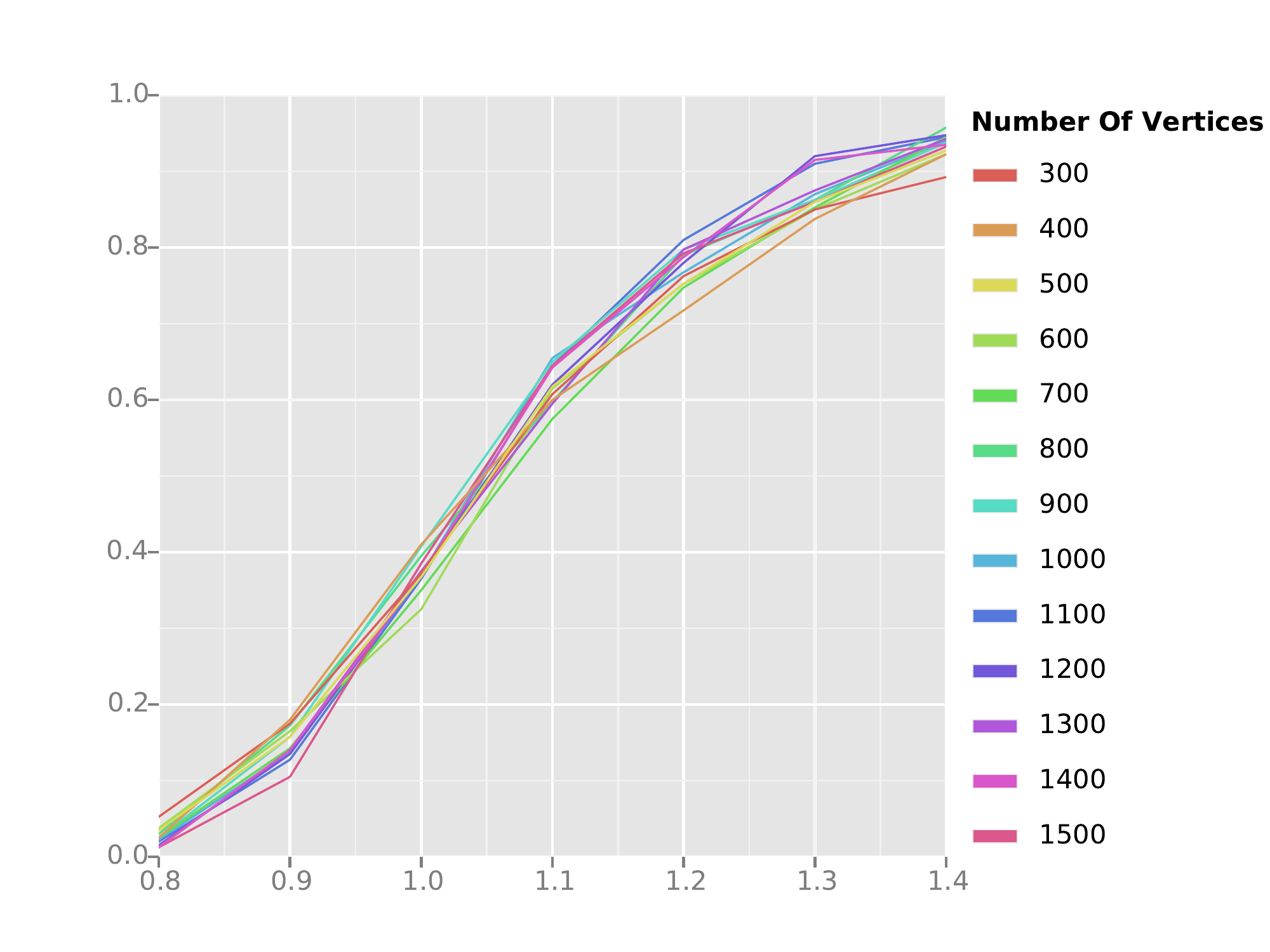}
\put(44,1){$\alpha$}
\put(-15,36){$\Pb$(connected)}
\end{overpic}
\caption{Experimental prevalence of connectedness at density $\alpha\frac{\log
n}{n}$.}\label{fig:conn_picture}
\end{figure}

\bibliographystyle{acm}

\end{document}